\documentclass[letterpaper,12pt]{amsart}
\title{Multigraded Commutative Algebra of Graph Decompositions}

\author{Alexander Engstr\"om}
\address{Department of Mathematics and Systems Analysis \\
Aalto University, Helsinki, Finland}
\email{alexander.engstrom@aalto.fi}
\author{Thomas Kahle}
\address{Zentrum Mathematik, TU M\"unchen \\ 85747 Garching b. M\"unchen, Germany}
\email{thomas.kahle@tum.de}
\author{Seth Sullivant}
\address{Department of Mathematics \\
North Carolina State University, Raleigh, NC, 27695}
\email{smsulli2@ncsu.edu}

\date{}

\usepackage{latexsym,array,delarray,amsthm,amssymb,epsfig}
\usepackage{color}
\usepackage{tikz} 
\usepackage{type1cm}
\usepackage{microtype}

\usepackage{url}
\usepackage{hyperref}
\definecolor{darkblue}{RGB}{0,0,160}
\hypersetup{
colorlinks,%
citecolor=black,%
filecolor=black,%
linkcolor=darkblue,%
urlcolor=darkblue
}

\theoremstyle{plain}
\newtheorem{thm}{Theorem}[section]
\newtheorem{lemma}[thm]{Lemma}
\newtheorem{prop}[thm]{Proposition}
\newtheorem{cor}[thm]{Corollary}
\newtheorem{conj}[thm]{Conjecture}

\theoremstyle{definition}
\newtheorem{defn}[thm]{Definition}

\newtheorem{ex}[thm]{Example}

\newtheorem{ques}[thm]{Question}

\theoremstyle{remark}

\newcommand{\zz}{\mathbb{Z}}
\newcommand{\nn}{\mathbb{N}}

\newcommand{\qq}{\mathbb{Q}}
\newcommand{\rr}{\mathbb{R}}

\newcommand{\kk}{\mathbb{K}}

\newcommand{\bfa}{\mathbf{a}}
\newcommand{\bfb}{\mathbf{b}}
\newcommand{\bfc}{\mathbf{c}}
\newcommand{\bfd}{\mathbf{d}}
\newcommand{\bfe}{\mathbf{e}}
\newcommand{\bff}{\mathbf{f}}
\newcommand{\bfg}{\mathbf{g}}
\newcommand{\bfh}{\mathbf{h}}

\newcommand{\bfr}{\mathbf{r}}
\newcommand{\bfs}{\mathbf{s}}

\newcommand{\bfu}{\mathbf{u}}
\newcommand{\bfv}{\mathbf{v}}
\newcommand{\bfw}{\mathbf{w}}

\newcommand{\cala}{\mathcal{A}}
\newcommand{\calb}{\mathcal{B}}
\newcommand{\calc}{\mathcal{C}}
\newcommand{\cald}{\mathcal{D}}
\newcommand{\calf}{\mathcal{F}}
\newcommand{\calg}{\mathcal{G}}
\newcommand{\calh}{\mathcal{H}}
\newcommand{\mm}{\mathcal{M}}

\newcommand{\spec}{\mathrm{Spec} \,}
\newcommand{\msp}{\times_{\nn\cala}}

\newcommand{\tfp}{\times_{\mathcal{A}}}
\newcommand{\tfpo}{\times_{\tilde{\mathcal{A}}}}
\newcommand{\tfps}{\times_{S}}
\newcommand{\<}{\left\langle}
\newcommand{\set}[1]{\left\lbrace #1 \right\rbrace} 
\renewcommand{\>}{\right\rangle}
\newcommand{\defas}{\mathrel{\mathop{:}}=}   
\DeclareMathOperator{\Lift}{Lift}
\DeclareMathOperator{\Quad}{Quad}
\DeclareMathOperator{\glue}{glue}
\DeclareMathOperator{\Glue}{Glue}
\DeclareMathOperator{\GGlue}{\mathbf{Glue}}
\DeclareMathOperator{\im}{im}
\DeclareMathOperator{\chara}{char}

\newcommand{\D}{\mathrm{D}} 

\newcommand\Perp{\protect\mathpalette{\protect\independenT}{\perp}}
\def\independenT#1#2{\mathrel{\rlap{$#1#2$}\mkern2mu{#1#2}}}
\newcommand{\ind}[2]{\left.#1 \Perp #2 \inD}
\newcommand{\inD}[1][\relax]{\def\argone{#1}\def\temprelax{\relax}
  \ifx\argone\temprelax\right.\else\,\middle|#1\right.{}\fi}

\begin{document}

\begin{abstract}
  The toric fiber product is a general procedure for gluing two ideals,
  homogeneous with respect to the same multigrading, to produce a new homogeneous
  ideal.  Toric fiber products generalize familiar constructions in commutative
  algebra like adding monomial ideals and the Segre product.  We describe how to
  obtain generating sets of toric fiber products in non-zero codimension and
  discuss persistence of normality and primary decompositions under toric fiber
  products.

  Several applications are discussed, including (a) the construction of Markov
  bases of hierarchical models in many new cases, (b) a new proof of the quartic
  generation of binary graph models associated to $K_{4}$-minor free graphs, and
  (c) the recursive computation of primary decompositions of conditional
  independence ideals.
\end{abstract}

\maketitle

\footnotesize

\setcounter{tocdepth}{1}

\normalsize

\section{Introduction}
Let $I$ and $J$ be ideals in polynomial rings $\kk[x]$ and $\kk[y]$, respectively,
that are both homogeneous with respect to a single grading by an affine
semigroup~$\nn\cala$.  The \emph{toric fiber product of $I$ and $J$}
(Definition~\ref{def:TFP}), denoted $I \times_\cala J,$ is a new ideal in a
usually larger polynomial ring~$\kk[z]$.  An important measure of complexity of
this operation is the \emph{codimension} of the product, defined as the rank of
the integer lattice $\ker \cala$.  In~\cite{Sullivant2007} the third author
introduced toric fiber products and proved that in the codimension zero case it is
possible to construct a generating set or Gr\"obner basis for $I \times_\cala J$
from generating sets or Gr\"obner bases of $I$ and~$J$.  In this case the algebra
and geometry is significantly simpler essentially because codimension zero toric
fiber products are \emph{multigraded Segre products}
(Definition~\ref{def:multisegre}), which share many nice properties with their
standard graded analogues.  Still in the codimension zero case, the geometry of
the toric fiber product can be understood quite explicitly in terms of
GIT~\cite{Mumford1994} (Propositions~\ref{prop:GIT} and~\ref{prop:msptfp}).  We
pursue this observation and show that (under mild assumptions on $\kk$) normality
persists (Theorem~\ref{thm:normal}).

The main goal of this paper, however, is to describe higher codimension toric
fiber products.  In Section~\ref{sec:pers-prim-decomp} we show that primary
decompositions persist in any codimension (Theorem~\ref{thm:primary-decomp}).  In
Section~\ref{sec:generators} we show how to construct generating sets of toric
fiber products in arbitrary codimension, but under some extra technical conditions
(Theorem~\ref{thm:cpp}). This generalizes the codimension one results on cut
ideals obtained by the first author in~\cite{Engstrom2008}.

The toric fiber product frequently appears in applications of combinatorial
commutative algebra, in particular in algebraic statistics~\cite{Engstrom2010,
  Sturmfels2008, Sturmfels2011}.  Typically in algebraic statistics, we are
interested in studying a family of ideals, where each ideal $I_{G}$ is associated
to a graph $G$ (or other combinatorial object, like a simplicial complex or a
poset).  If the graph has a decomposition into two simpler graphs $G_{1}$ and
$G_{2}$, we would like to show that the ideal $I_{G}$ has a decomposition into the
two ideals $I_{G_{1}}$ and~$I_{G_{2}}$.  If we can identify $I_{G}$ as a toric
fiber product $I_{G_{1}} \tfp I_{G_{2}}$, then difficult algebraic questions for
large graphs reduce to simpler problems on smaller graphs.  Our inspiration comes
from structural graph theory, where the imposition of forbidden substructures
often implies that a graph has a specific kind of structural decomposition into
simple pieces.  In Section~\ref{sec:markov} we pursue the analogy to the theory of
forbidden minors~\cite{Robertson2004} by exhibiting minor-closed classes of graphs
with certain degree bounds on their Markov bases.

Before proving our main theoretical results in
Sections~\ref{sec:tfpandmsp}--\ref{sec:generators}, we motivate our study with
several examples from algebraic statistics.  Sections~\ref{sec:markov}
and~\ref{sec:ci} contain new applications to the construction of Markov bases of
hierarchical models, and to the study of primary decompositions of conditional
independence ideals.


\subsection{Hierarchical models} \label{sec:2:hierarchical}

Hierarchical statistical models are used to analyze associations between
collections of random variables.  If the random variables are discrete, these
models are toric varieties, and hence their vanishing ideals are toric ideals.
Their binomial generators---known as \emph{Markov bases}---are useful for
performing various tests in statistics~\cite{Diaconis1998, Drton2009}.  From the
algebraic standpoint, they are binomial ideals with a specific combinatorial
parametrization in terms of a simplicial complex.

Let $\Gamma \subseteq 2^{V}$ be a simplicial complex on a finite set $V$ and $d
\in \zz^{V}_{\geq 2}$.  Let ${\rm facet}(\Gamma)$ be the set of maximal faces
of~$\Gamma$.  For an integer $n$, let $[n] = \{1,2, \ldots, n\}$.  For $F
\subseteq V$ let $d_{F} = (d_{v})_{v \in F}$ and let $\D_{F} = \prod_{v \in F}
[d_{v}]$.  For $i \in \D_{V}$ and $F \subseteq V$ let $i_{F} = (i_{v})_{v \in F}$
be the restriction.  For each $F \in {\rm facet}(\Gamma)$ and $i \in \D_{F}$, let
$a^{F}_{i}$ be an indeterminate.  For each $i \in \D_{V}$, let $p_{i}$ be another
indeterminate.  The toric ideal $I_{\Gamma, d}$ of the hierarchical model for
$(\Gamma,d)$ is the kernel of the $\kk$-algebra homomorphism
\[
\phi_{\Gamma,d} : \kk[p_{i}: i \in \D_{V}] \rightarrow \kk[ a^{F}_{j} : F \in {\rm
  facet}(\Gamma), j \in \D_{F} ] \qquad p_{i} \mapsto   \prod_{F \in {\rm facet}(\Gamma)}  a^{F}_{i_{F}}.
\]
A fundamental problem of algebraic statistics is to determine generators
for~$I_{\Gamma,d}$.  Results in this direction usually depend on special
properties of $\Gamma$ and~$d$.  An example is the following theorem of Kr\'al,
Norine, and Pangr\'ac \cite{Kral2008}, which is also a corollary to our results in
Section~\ref{sec:seriespar}:

\begin{thm}\label{thm:kral}
  Let $d_{i} = 2$ for all $i \in V$ and let $\Gamma$ be a graph with no $K_{4}$
  minors.  Then $I_{\Gamma,d}$ is generated by binomials of degrees two and four.
\end{thm}

Combining our techniques with results from~\cite{Hillar2009}, we can also make
statements about the asymptotic behavior as the $d_{i}$ grow.  For instance, let
$F \subseteq V$ be an independent set of $\Gamma$ and consider $I_{\Gamma,d}$ as
$d_{i}$ tend to infinity for $i \in F$, while the remaining $d_{i}$ are fixed.  In
this case, there is a bound $M(\Gamma, d_{V \setminus F})$ for the degrees of
elements in minimal generating sets of
$I_{\Gamma,d}$.  
Our techniques allow us to determine the values of $M(\Gamma, d_{V \setminus F})$,
which were previously known only for reducible models or when $F$ is a
singleton~\cite{Hosten2007}.  Here is a simple example of how to apply
Theorem~\ref{thm:s6main}.

\begin{ex}\label{ex:hillar}
  Let $\Gamma = [12][13][24][34]$ be a four cycle, $F = \{1,4\}$, and $d_{\{2,3\}}
  = (2,2)$.  The toric ideal $I_{\Gamma,d}$ is a codimension one toric fiber
  product 
  and its minimal generating set 
  consists of the following four types of binomials, written in tableau notation
  (a common notation, explained below Theorem~\ref{thm:Diaconis}):
$$
\begin{bmatrix}
i_1 &  2 & 2  & l_1 \\
i_1 &  1 & 1  & l_2 \\
i_2 & 2 & 1  & l_3 \\
i_2 & 1 & 2 & l_4  
\end{bmatrix} -
\begin{bmatrix}
i_1 &  2 & 1  & l_3 \\
i_1 &  1 & 2  & l_4 \\
i_2 & 2 & 2  & l_1 \\
i_2 & 1 & 1 & l_2  
\end{bmatrix}
 \quad \quad \quad \quad
\begin{bmatrix}
i_1 &  2 & 2  & l_1 \\
i_2 &  1 & 1  & l_1 \\
i_3 & 2 & 1  & l_2 \\
i_4 & 1 & 2 & l_2  
\end{bmatrix} -
\begin{bmatrix}
i_3 &  2 & 1  & l_1 \\
i_4 &  1 & 2  & l_1 \\
i_1 & 2 & 2  & l_2 \\
i_2 & 1 & 1 & l_2  
\end{bmatrix}
$$
$$
\begin{bmatrix}
i_1 &  j & k  & l_1 \\
i_2 &  j & k  & l_2   
\end{bmatrix} -
\begin{bmatrix}
i_1 &  j & k  & l_2 \\
i_2 &  j & k  & l_1   
\end{bmatrix}
 \quad \quad \quad \quad
\begin{bmatrix}
i &  2 & 2  & l \\
i &  1 & 1  & l  
\end{bmatrix} -
\begin{bmatrix}
i &  2 & 1  & l \\
i &  1 & 2  & l 
\end{bmatrix}
$$
where $i,i_1,i_2,i_3,i_4 \in [d_{1}]$, $j, k \in [2]$, $l, l_1, l_2, l_3, l_4 \in
[d_{4}]$.  In particular, $M(\Gamma, d_{V \setminus F}) = 4$.
\end{ex}


\subsection{Conditional independence}\label{sec:introCI}

If $G$ is a graph on $V$, then its clique complex defines a hierarchical model as
in the previous section.  Probability distributions in this hierarchical model
satisfy certain conditional independence statements associated to the
graph~\cite{Lauritzen1996}.  One may ask which other distributions outside the
hierarchical model also satisfy the conditional independence constraints, and
algebraic statistics allows one to characterize these distributions.  Consider
again the polynomial ring $\kk[p_{i} : i \in \D_{V}]$ with one indeterminate for
each elementary probability.  If $A,B,C \subset V$ is a~partition of $V$,
i.e.~pairwise disjoint with $A\cup B\cup C = V$, the \emph{conditional
  independence (CI)-statement} $\ind{A}{B}[C]$ encodes that the random variables
in $A$ are independent of the random variables in $B$, given the values of the
random variables in $C$.  Distributions satisfying this constraint form a
hierarchical model, which arises from the largest simplicial complex on $V$ not
containing $\set{i,j}$ for any $i\in A, j\in B$.  Its toric ideal is
denoted~$I_{\ind{A}{B}[C]}$.  A \emph{conditional independence model} usually
contains several statements and one is led to consider intersections of toric
varieties.
Our main interest is in the \emph{global Markov ideal} of a graph $G$, which is
the sum of the toric ideals $I_{\ind{A}{B}[C]}$ for all $A,B,C$ forming a
partition of $V$ such that $C$ separates $A$ and $B$ in~$G$.  Our goal is to
determine primary decompositions and as always we want to employ the toric fiber
product machinery to split the problem into several easier problems.

\begin{ex}\label{ex:3squares}
  Let $G$ be the binary global Markov ideal of the graph in
  Figure~\ref{fig:3squares}.
  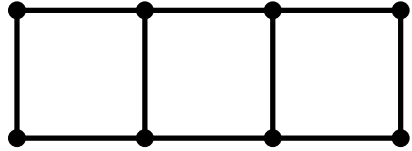
\begin{figure}[ht]
    \centering 
    \begin{tikzpicture}[scale=1.7, line width=.7mm]
      \foreach \i in {0,...,3} { \path (\i,0) coordinate (X\i); \path (\i,1)
        coordinate (Y\i); \fill (X\i) circle (2pt); \fill (Y\i) circle (2pt);
        \draw (X\i) -- (Y\i); } \draw (X0) -- (X3); \draw (Y0) -- (Y3);
    \end{tikzpicture}
    \caption{\label{fig:3squares} Three squares glued along edges}
  \end{figure}
  Since it decomposes as three squares glued along edges,
  Theorem~\ref{thm:primary-decomp} and Corollary~\ref{thm:irredundant} reconstruct
  the primary decomposition from that of the CI-ideal of a square.  Our results
  also show that the corresponding CI-ideal is radical, as it is composed of
  graphs with radical CI-ideals.  In total it is the intersection of $729=9^{3}$
  prime ideals.
\end{ex}

A systematic check of all graphs with at most five vertices and with $d_{v} = 2$
for all $v \in V$ found no examples of a non-radical global Markov ideal.  This
limited computational evidence motivates the following question:
\begin{ques}
  Are global Markov ideals always radical?
\end{ques}

The answer to this question is negative.  More than a year after first submission
of the present paper, Kahle, Rauh, and Sullivant showed that the global Markov
ideal of $K_{3,3}$ is not radical~\cite{KRS12}.


\section{Toric fiber products and multigraded Segre products}\label{sec:tfpandmsp}

Let $r > 0$ be a positive integer and $s,t \in \zz^r_{> 0}$ be two vectors of
positive integers.  Let
$$
\kk[x] = \kk[x^i_j : i \in [r], j \in [s_i]] \quad \quad \mbox{ and } \quad \quad
\kk[y] = \kk[y^i_k : i \in [r], k \in [t_i]]
$$
be multigraded polynomial rings subject to a multigrading
$$
\deg(x^i_j) = \deg(y^i_k) = \bfa^i \in \zz^d.
$$
We assume throughout that there exists a vector $\omega \in \qq^d$ such that
$\omega^T\bfa^i = 1$ for all~$i$.  This implies that ideals homogeneous with
respect to the multigrading are homogeneous with respect to the standard coarse
grading.  Let $\cala = \{ \bfa^1, \ldots, \bfa^r \}$ and let $\nn\cala$ be the
affine semigroup generated by~$\cala$.  If $I \subset \kk[x]$ and $J
\subset\kk[y]$ are $\nn\cala$-graded ideals, the quotient rings $R = \kk[x]/I$ and
$S = \kk[y]/J$ are also $\nn\cala$-graded.  Let
$$
\kk[z] = \kk[z^i_{jk} : i \in [r], j \in [s_i], k \in [t_i]]
$$
and let $\phi_{I,J} : \kk[z] \rightarrow R \otimes_\kk S$ be the $\kk$-algebra
homomorphism such that $z^i_{jk} \mapsto x^i_j \otimes y^i_k.$

\begin{defn}\label{def:TFP}
  The \emph{toric fiber product $I\tfp J$} of $I$ and $J$ is the kernel of
  $\phi_{I,J}$:
$$I\tfp J = \ker \phi_{I,J}.$$
The \emph{codimension} of the toric fiber product is the dimension of the space of
linear relations among~$\cala$.
\end{defn}

We can also define the $\kk$-algebra homomorphism $\phi : \kk[z] \rightarrow
\kk[x] \otimes_\kk \kk[y] = \kk[x,y]$ by $z^i_{jk} \mapsto x^i_j y^i_k$.  Then the
toric fiber product is the ideal $I\tfp J = \phi^{-1}(I + J)$.

\subsection{The geometry of toric fiber products}
If $I \times_\cala J$ is a codimension zero toric fiber product, the relation
between the schemes $\spec( \kk[x]/I)$, $\spec(\kk[y]/J)$ and $\spec( \kk[z]/(I
\times_\cala J))$ can be explained in the language of GIT (geometric invariant
theory) quotients.  Since $I$ and $J$ are homogeneous with respect to the grading
by $\cala$, both $\spec( \kk[x]/I)$ and $\spec(\kk[y]/J)$ have an action of a
$(\dim \cala - 1)$-dimensional torus~$T$.  Thus the product scheme $\spec(
\kk[x]/I) \times \spec(\kk[y]/J)$ possesses an action of $T$ via $t \cdot (x,y) =
(tx, t^{-1} y)$.

\begin{prop}\label{prop:GIT}
  If $\kk$ is algebraically closed and $\cala$ is linearly independent, then
  \begin{equation*}
    \spec(\kk[z]/(I \times_\cala J))  \cong \left( \spec( \kk[x]/I) \times \spec(\kk[y]/J) \right)//T.  
  \end{equation*}
\end{prop}

\begin{proof}
  If $\kk$ is algebraically closed, then
  \begin{equation*}
    \spec( \kk[x]/I) \times \spec(\kk[y]/J)
    = \spec( \kk[x]/I \otimes_\kk \kk[y]/J).
  \end{equation*}
  Let $R = \kk[x]/I$ and $S = \kk[x]/J$.  Both $R$ and $S$ are $\nn\cala$-graded,
  so we can write $R = \oplus_{\bfa \in \nn\cala} R_\bfa$, $S = \oplus_{\bfa \in
    \nn\cala} S_\bfa$, and
  \begin{equation*}
    \kk[x]/I \otimes_\kk \kk[y]/J  =  \oplus_{\bfa \in \nn\cala, \bfb \in \nn\cala} R_\bfa \otimes_\kk S_\bfb,
  \end{equation*}
  where the $\zz \cala$ degree of $R_\bfa \otimes_\kk S_\bfb$ is $\bfa - \bfb$.
  The invariant ring of the torus action is the degree ${\bf 0}$ part, which is
  $\oplus_{\bfa \in \nn\cala} R_\bfa \otimes_\kk S_\bfa$.  The proof is complete
  once we show that
  \begin{equation} \label{eq:tfpmsp}
    \kk[z]/(I \times_\cala J)  \cong \oplus_{\bfa \in \nn\cala}  R_\bfa \otimes_\kk S_\bfa,
  \end{equation}
  since then the spectra must be the same.
  The toric fiber product $I \times_\cala J$ is the kernel of the ring
  homomorphism
  \begin{equation*}
    \phi:  \kk[z] \rightarrow R \otimes_\kk S, \quad \quad z^i_{jk} \rightarrow x^i_{j} \otimes y^i_k,
  \end{equation*}
  thus the first isomorphism theorem asserts $\kk[z]/(I \times_\cala J) \cong \im
  \phi$.  Since $\deg(x^i_j) = \deg (y^i_k) = \bfa_{i}$, $\im \phi$ is a
  subalgebra of $\oplus_{\bfa \in \nn\cala} R_\bfa \otimes_\kk S_\bfa$.  We need
  to show that $\phi$ surjects onto it.
  As algebras, $R$ is generated by $\oplus_{\bfa \in \cala} R_\bfa$ and $S$ is
  generated by $\oplus_{\bfa \in \cala} S_\bfa$.  Now let $x^\bfu \otimes y^\bfv$ be a
  monomial in some $R_\bfa \otimes_\kk S_\bfa$.  Since $\cala = \{\bfa_1, \ldots,
  \bfa_n\}$ is linearly independent, there is a unique way to write $\bfa =
  \sum_{i =1}^n \lambda_i \bfa_i$ with $\lambda_{i} \in \nn$.  Thus
  \begin{equation*}
    x^\bfu  =  \prod_{r =1}^{\lambda_1} x^1_{j_{1r}}  \cdots \prod_{r =1}^{\lambda_n}
    x^n_{j_{nr}} \quad \text{ and } \quad y^\bfv = \prod_{r =1}^{\lambda_1} y^1_{k_{1r}}
    \cdots \prod_{r =1}^{\lambda_n} y^n_{k_{nr}}.
  \end{equation*}
  So we have
  \begin{equation*}
    x^\bfu \otimes y^\bfv  =  
    \prod_{r =1}^{\lambda_1} x^1_{j_{1r}} \otimes y^1_{k_{1r}}  \cdots \prod_{r =1}^{\lambda_n} x^n_{j_{nr}} \otimes y^n_{k_{nr}}
  \end{equation*}
  and this monomial is in the subring generated by $\oplus_{\bfa \in \cala} R_\bfa
  \otimes_\kk S_\bfa$.  Since the monomials span the entire ring $\oplus_{\bfa \in
    \nn\cala} R_\bfa \otimes_\kk S_\bfa$ as a vector space, every element in
  $\oplus_{\bfa \in \nn\cala} R_\bfa \otimes_\kk S_\bfa$ is in $\im \phi \cong
  \kk[z]/(I \times_\cala J)$, which completes the proof.
\end{proof}

The assumption of linear independence is essential for the proof of
Proposition~\ref{prop:GIT} and the statement is no longer true if $\cala$ is
linearly dependent.  We always have
\[
\left( \spec(\kk[x]/I) \times \spec(\kk[y]/J) \right) // T = \spec(
\bigoplus_{\bfa \in \nn\cala} R_\bfa \otimes_\kk S_\bfa )
\]
but \eqref{eq:tfpmsp} fails.  Indeed, $\kk[z]/(I \times_\cala J)$ is a strict
subset of $\oplus_{\bfa \in \nn\cala} R_\bfa \otimes_\kk S_\bfa$ when $\cala$ is
linearly dependent.  While not, in general, a toric fiber product, this ring and
the associated GIT quotient do arise in algebraic geometry, in particular in the
work of Buczynska~\cite{Buczynska2010} and Manon~\cite{Manon2009}.  Because of its
appearance in other contexts, we feel that this object is worthy of its own
definition.

\begin{defn} \label{def:multisegre}
  Let $R$ and $S$ be two rings graded by a common semigroup~$\nn\cala$.  The
  \emph{multigraded Segre product} is 
  \begin{equation*}
    R \times_{\nn\cala} S  =  \oplus_{\bfa \in \nn\cala} R_\bfa \otimes_\kk S_\bfa.
  \end{equation*}
\end{defn}

With this new definition, Proposition~\ref{prop:GIT} is equivalent to the
statement:
\begin{prop} \label{prop:msptfp} If $\cala$ is linearly independent, then 
  \begin{equation*}
    \kk[x]/I \times_{\nn\cala} \kk[y]/J  \cong  \kk[z]/(I \times_\cala J).
\end{equation*}
\end{prop}

\subsection{Persistence of normality} \label{sec:normal} One of the most basic
questions about an ideal $I$ in a ring $R$ is whether or not the quotient $R/I$ is
normal.  When $I$ is a toric ideal, $\kk[x]/I$ is an affine semigroup ring and
normality can be characterized in terms of the semigroup having no holes.  In
algebraic statistics, normality implies favorable properties of sampling
algorithms for contingency tables~\cite{Chen2006, Takemura2008}.  In this section
we show that normality persists under codimension zero toric fiber products.  We
only treat the case of (not necessarily toric) prime ideals, which suffices in
many situations (see for instance~\cite[Proposition~2.1.16]{SwaHu}).

\begin{thm} \label{thm:normal} Let $I$ and $J$ be homogeneous prime
  $\nn\cala$-graded ideals, with $\cala$ linearly independent, and suppose that
  $\kk[x]/I$ and $\kk[y]/J$ are normal domains (that is, integrally closed in
  their field of fractions).  If $\kk$ is algebraically closed, then $\kk[z]/ (I
  \tfp J)$ is normal.
\end{thm}
The assumption that $\kk$ is algebraically closed is needed to ensure that
$\kk[z]/ (I \tfp J)$ is a domain.  This holds more generally if $I$ and $J$ are
geometrically prime (see Theorem~\ref{thm:primary-decomp}).  If this is given, the
field assumption can be weakened to $\kk$ being a \emph{perfect field}, that is a
field $\kk$ such that either $\chara (\kk) = 0$ or $\chara(\kk) = p$ and $\kk =
\set{a^{p} : a\in \kk}$.  The proof of Theorem~\ref{thm:normal} is based on the
following observation which is easy and independent of the codimension of~$\cala$.
\begin{lemma}\label{lemma:directsummand}
  The multigraded Segre product 
  is a direct summand of the tensor product $R \otimes_{\kk} S$ (as a module over
  the subring).
\end{lemma}
\begin{proof}
  The inclusion $0 \to \bigoplus_{a\in\nn\cala} R_{a} \otimes_{\kk} S_{a} \to
  \bigoplus_{a\in\nn\cala}\bigoplus_{b\in\nn\cala} R_{a} \otimes_{\kk} S_{b}$
  splits via the $(\bigoplus_{a\in\nn\cala} R_{a} \otimes_{\kk} S_{a})$-module
  homomorphism that maps $x^{i}_{j}\otimes y^{l}_{k}$ to itself if $\bfa_{i} =
  \bfa_{l}$ and zero otherwise.
\end{proof}
We anticipate that Lemma~\ref{lemma:directsummand} will be useful in relating
properties of multigraded Segre products to those of the factors.  For instance, a
careful analysis of the Castelnuovo--Mumford regularity would be interesting, but
is beyond the scope of this paper.  We apply the lemma to prove persistence of
normality in codimension zero.  Note that the codimension requirement enters
because only if $\cala$ is linearly independent, Lemma~\ref{lemma:directsummand}
gives us a handle on the toric fiber product.
\begin{proof}[Proof of Theorem~\ref{thm:normal}] Let $R = \kk[x]/I$ and
  $S=\kk[y]/J$.  It is easy to see directly (and also follows from
  Theorem~\ref{thm:primary-decomp} below) that $\kk[z]/(I \tfp J)$ is a domain,
  given that $\kk$ is algebraically closed.  An algebraically closed field is
  perfect and therefore, if $R$ and $S$ are normal, then $R\otimes_{\kk}S$ is
  normal.  This follows from Serre's criterion and~\cite[Theorem~6]{Tousi2003}.
  Since a direct summand of a normal domain is normal,
  Lemma~\ref{lemma:directsummand} completes the proof.
\end{proof}

The main case of interest for our applications is when the ideals $I$ and $J$ are
toric ideals and various special cases have been proved in the algebraic
statistics literature.  For example, Ohsugi~\cite{Ohsugi2010} proves this for cut
ideals, Sullivant~\cite{Sullivant2010} for hierarchical models, and Micha\l ek
\cite{Michalek2010} for group-based phylogenetic models.  The proofs of these
results are essentially the same, and consists of analyzing a toric fiber product
of the grading semigroup.  We introduce this setting now.

\subsection{Fiber products of vector configurations}\label{sec:vector}

If $I$ and $J$ are toric ideals, then $I \tfp J$ is also a toric ideal.  The
corresponding vector configuration arises from taking the fiber product of the two
vector configurations corresponding to $I$ and~$J$.  Let $\calb = \{ \bfb^i_j : i
\in [r], j \in [s_i] \} \subseteq \zz^{d_1}$ and $\calc = \{ \bfc^i_k : i \in [r],
k \in [t_i] \} \subseteq \zz^{d_2}$ be two vector configurations.  As necessary,
we consider $\calb$ and $\calc$ as collections of vectors or as matrices.  These
vector configurations define toric ideals $I_\calb \subseteq \kk[x]$ and $I_\calc
\subseteq \kk[y]$ by
\begin{equation*}
I_\calb = \<  x^\bfu - x^\bfv :  \calb \bfu = \calb \bfv \>  \qquad \text{ and }
\qquad I_\calc = \< y^\bfu - y^\bfv :  \calc \bfu = \calc \bfv \>.
\end{equation*}
To say that $I_\calb$ and $I_\calc$ are homogeneous with respect to the grading by
$\cala$ with $\deg(x^i_j) = \deg(y^i_k) = \bfa^i$ is to say that there are linear
maps $\pi_1 : \zz^{d_1} \rightarrow \zz^e$ and $\pi_2: \zz^{d_2} \rightarrow
\zz^e$ such that $\pi_1(\bfb^i_j) = \bfa^i$ for all $i$ and $j$ and
$\pi_2(\bfc^i_k) = \bfa^i$ for all $i$ and~$k$.  The new vector configuration that
arises in this case is the fiber product of the vector configurations.
\begin{equation*}
\calb \tfp \calc =  \{ (\bfb^i_j, \bfc^i_k) \in \zz^{d_1 + d_2} :  i \in [r], j \in [s_i], k \in [t_i] \}.
\end{equation*}
The notation is set up so that the toric fiber product $I_\calb \tfp I_\calc$ is
the toric ideal
\[
I_\calb \tfp I_\calc = I_{\calb \tfp \calc}  =  \< z^\bfu - z^\bfv :  (\calb \tfp
\calc)\bfu = (\calb \tfp \calc) \bfv \>.
\]
Indeed, if $\kk[s]$ and $\kk[t]$ are polynomial rings, and
\begin{gather*}
\phi:  \kk[x] \to \kk[s] \quad x^i_j  \mapsto f^i_j(s)\\
\psi:  \kk[y] \to  \kk[t] \quad y^i_k  \mapsto g^i_k(t)
\end{gather*}
are $\kk$-algebra homomorphisms, then we can form the toric fiber product
homomorphism
\begin{equation*}
  \phi \times_\cala \psi :  \kk[z]  \to  \kk[s,t] \quad z^i_{j,k}  \mapsto  f^i_j(s) g^i_j(t).
\end{equation*}
If $I = \ker \phi, J = \ker \psi$ and both ideals are homogeneous with respect to
the grading by $\cala$, then $I \times_\cala J = \ker (\phi \times_\cala \psi)$.  In
the toric case, when $\phi,\psi$ are monomial homomorphisms, it is easy to see
that $\calb \tfp \calc$ defines the toric fiber product homomorphism.

In most cases our interest is in the ideal $I_\calb \tfp I_\calc = I_{\calb \tfp
  \calc}$ and not the specific vector configuration.  A useful technique is to
modify the vector configuration $\calb \tfp \calc$ to any other set of vectors
with the same kernel, without changing the toric ideal.  For example, we could
also use the vector configuration
\begin{equation*}
\calb \tfp \calc =  \{ (\bfb^i_j,\bfa^i,  \bfc^i_k) \in \zz^{d_1 + e + d_2} :  i \in [r], j \in [s_i], k \in [t_i] \}.
\end{equation*}


\section{Persistence of primary decomposition}
\label{sec:pers-prim-decomp}
Primary decompositions of toric fiber products consist of toric fiber products of
primary components.
To state the result, recall that an ideal is \emph{geometrically primary} if it is
primary over any algebraic extension of the coefficient field.

\begin{thm}
  \label{thm:primary-decomp}
  Let $I \subseteq \kk[x]$ and $J \subseteq \kk[y]$ be $\cala$-homogeneous ideals.
  Let $I = I_1 \cap \cdots \cap I_k$ and $J = J_1 \cap \cdots \cap J_l$ be primary
  decompositions of $I$ and $J$ such that all ideals $I_{i}$ and $J_{j}$ are
  homogeneous with respect to~$\cala$.  Then
\begin{equation} \label{eq:primary}
I \times_\cala J  =  \cap_{i = 1}^k \cap_{j = 1}^l  I_i \times_\cala J_j.
\end{equation}
If, in addition, the ideals $I_i$ and $J_j$ are all geometrically primary, then
(\ref{eq:primary}) is a primary decomposition of $I \times_\cala J$.
\end{thm}
\begin{proof}
  First we show that the decomposition is valid.  This follows if we show that for
  all $\nn\cala$ homogeneous ideals $I_1, I_2 \in \kk[x]$ and $J \in \kk[y]$,
$$
(I_1 \cap I_2) \tfp J  =  (I_1 \tfp J) \cap (I_2 \tfp J). 
$$
Let $\phi : \kk[z] \rightarrow \kk[x] \otimes_\kk \kk[y]$ be the $\kk$-algebra
homomorphism such that $z^i_{jk} \mapsto x^i_j \otimes y^i_k$.  A polynomial $f$
belongs to a toric fiber product $I \tfp J$ if and only if $\phi(f) \in I + J
\subseteq \kk[x] \otimes_\kk \kk[y] = \kk[x,y]$.  Thus
\begin{eqnarray*}
 f \in (I_1 \cap I_2) \tfp J & \Leftrightarrow &  \phi(f) \in   (I_1 \cap I_2)  + J \\
& \Leftrightarrow & \phi(f) \in  (I_1 + J) \cap (I_2 + J) \\
& \Leftrightarrow & f \in (I_1 \tfp J) \cap (I_2 \tfp J),
\end{eqnarray*}
where the second equivalence is because $I_{i}$ and $J$ are ideals in disjoint
sets of variables.

For the second claim, since $I_i \tfp J_j$ is the inverse image of $I_i + J_j$,
and inverse images of primary ideals are primary, it suffices to show, for any
geometrically primary ideals $I \subseteq \kk[x]$ and $J \subseteq \kk[y]$, that
$I+J \subseteq \kk[x,y] $ is geometrically primary.  First, note that the
statement clearly holds if $I$ and $J$ are geometrically prime ideals, since the
join of two irreducible varieties is irreducible.  The proof of Proposition 1.2
(iv) in \cite{Simis2000} contains the cases of geometrically primary ideals.
\end{proof}

\begin{thm}\label{thm:irredundant}
Suppose that $\cala$ is linearly independent.  Then
the decomposition 
\begin{equation} \label{eq:primaryAgain}
I \times_\cala J  =  \cap_{i = 1}^k \cap_{j = 1}^l  I_i \times_\cala J_j
\end{equation}
is irredundant if and only if for all $i_{1}, i_{2} \in [k]$ and $j_{1}, j_{2} \in
[l]$ with $i_{1} \neq i_{2}$ or $j_{1} \neq j_{2}$ either:
\begin{itemize}
\item there exists $\bfa \in \nn\cala$ such that $(I_{i_1})_\bfa \not\subseteq
  (I_{i_2})_\bfa$ and $(J_{j_{2}})_\bfa \neq \kk[y]_\bfa$, or
\item there exists $\bfb \in \nn\cala$ such that $(J_{j_1})_\bfb \not\subseteq
  (J_{j_2})_\bfb$ and $(I_{i_{2}})_\bfb \neq \kk[x]_\bfb$.
\end{itemize}
\end{thm}

\begin{proof}
  To deal with redundancy of the decomposition, we must describe conditions on
  $I, K \subseteq \kk[x]$ and $J,L \subseteq \kk[y]$ that imply $I \tfp J
  \subseteq K \tfp L$.
  Let $R =\kk[x]/I$, $S = \kk[y]/J$, $R' = \kk[x]/K$, and $S' = \kk[y]/L$.  Since
  $\cala$ is linearly independent, the rings $\kk[z]/ (I \tfp J)$ and $\kk[z]/ (K
  \tfp L)$ are multigraded Segre products.  So $I \tfp J \subseteq K \tfp L$ if
  and only if $R' \msp S'$ is a quotient of $R \msp S$ by the ideal generated by
  the image of $K \tfp L$ in $R \msp S$.  On the level of the homogeneous
  components, we require that $R'_\bfa \otimes_\kk S'_\bfa = R_\bfa \otimes_\kk
  S_\bfa / (K \tfp L)_\bfa$, as $\kk$-vector spaces.  There are two ways that
  $R'_\bfa \otimes_\kk S'_\bfa$ could be a quotient of $R_\bfa \otimes_\kk
  S_\bfa$.  If $I_\bfa \subseteq K_\bfa$ and $J_\bfa \subseteq L_\bfa$, then $(I
  \tfp J)_{\bfa} \subseteq (K \tfp L)_{\bfa}$, in which case we have the desired
  quotient.  The second way is if the tensor product $R'_\bfa \otimes_\kk S'_\bfa
  = \{0 \}$, which happens if and only if either $R'_\bfa$ or $S'_\bfa$ is
  $\{0\}$.  On the level of ideals, this happens if and only if either $K_\bfa =
  \kk[x]_\bfa$ or $L_\bfa = \kk[y]_\bfa$.

  The decomposition~\eqref{eq:primaryAgain} is redundant if and only if there are
  $i_1, i_2$ and $j_1, j_2$ where $I_{i_1} \tfp J_{j_1} \subseteq I_{i_2} \tfp
  J_{j_2}$ (where one of $i_1 = i_2$ and $j_1 = j_2$ is allowed, but not both).
  Now $I_{i_1} \tfp J_{j_1} \subseteq I_{i_2} \tfp J_{j_2}$ if and only if for all
  $\bfa \in \nn\cala$, $(\kk[x]/I_{i_2})_\bfa \otimes_\kk (\kk[y]/J_{j_2})_\bfa$
  is a quotient of $(\kk[x]/I_{i_1})_\bfa \otimes_\kk (\kk[y]/J_{j_1})_\bfa$.
  This happens if and only if for each $\bfa \in \nn\cala$ the condition in the
  previous paragraph is satisfied.  Thus, $I_{i_1} \tfp J_{j_1} \not\subseteq
  I_{i_2} \tfp J_{j_2}$ if and only if the negation of this condition holds.
  Choosing $\bfa$ from the first condition of the theorem with respect to $j =
  j_2$, yields the desired non-containment in the case $i_1 \neq i_2$.  If $i_1 =
  i_2$ and $j_1 \neq j_2$, we choose $\bfb$ from the second condition of the
  theorem with respect to $i = i_1$.  This proves the sufficiency of the
  conditions.

  The two conditions are necessary since the first is necessary for $I_{i_1} \tfp
  J_{j} \not\subseteq I_{i_2} \tfp J_{j}$, while the second is necessary for $I_i
  \tfp J_{j_1} \not\subseteq I_i \tfp J_{j_2}$.
\end{proof}

\begin{cor} \label{cor:pers-prim-decomp-1} Let $\cala$ be linearly independent.
  Suppose that $I = I_1 \cap \cdots \cap I_k$ and $J = J_1 \cap \cdots \cap J_l$
  are $\cala$ homogeneous irredundant primary decompositions of $I$ and $J$ into
  geometrically primary ideals, and that for each $i \in [k]$, $j \in [l]$, and
  $\bfa \in \nn\cala$, neither $(I_i)_\bfa = \kk[x]_\bfa$ nor $(J_j)_\bfa =
  \kk[y]_\bfa$.  Then
  \begin{equation*}
    I \times_\cala J  =  \cap_{i = 1}^k \cap_{j = 1}^l  I_i \times_\cala J_j
\end{equation*}
is an irredundant primary decomposition of $I \tfp J$.
\end{cor}

\begin{proof}
  We combine Theorems~\ref{thm:primary-decomp} and~\ref{thm:irredundant}.  Since
  the ideals $I_i$ and $J_j$ are all geometrically primary, the decomposition of
  $I \tfp J$ is a primary decomposition.  Since the decomposition of $I$ is
  irredundant, for each $i_1 \neq i_2$ there exists a $\bfa \in \nn\cala$ such
  that $(I_{i_1})_\bfa \not\subseteq (I_{i_2})_\bfa$ and, by assumption, for all
  $j$ $(J_j)_\bfa \neq \kk[y]_\bfa$.  Similarly, the decomposition of $J$ is
  irredundant, for each $j_1 \neq j_2$ there exists a $\bfb \in \nn\cala$ such
  that $(J_{j_1})_\bfb \not\subseteq (J_{j_2})_\bfb$ and, by assumption, for all
  $i$, $(I_i)_\bfb \neq \kk[y]_\bfb$.  This implies that the decomposition is
  irredundant.
\end{proof}

To apply Corollary~\ref{cor:pers-prim-decomp-1} iteratively, we need to control
when its hypotheses are preserved.
\begin{lemma}
  \label{lem:two-gradings}
  Let $\mathcal{A}$ be linearly independent, and let $\mathcal{B}$ induce a
  grading on $\kk[x]$ such that
  \begin{itemize}
  \item for all $\mathbf{b} \in \nn\mathcal{B}$ $(I)_{\mathbf{b}} \neq
    \kk[x]_{\mathbf{b}}$, and
  \item for all $\bfa \in \nn\mathcal{A}$ $(J)_{\bfa} \neq \kk[y]_{\bfa}$.
  \end{itemize}
  In this case $(I \times_{\mathcal{A}} J)_{\mathbf{b}} \neq \kk[z]_{\mathbf{b}}$
  for all $\mathbf{b}\in\nn\mathcal{B}$.
\end{lemma}
\begin{proof}
  Let $R = \kk[x]/I$ and $S=\kk[y]/J$.  We decompose the $\mathcal{A}$-graded
  parts of $R$ into their $\mathcal{B}$-graded parts. The conclusion is equivalent
  to the statement that in \[\bigoplus_{(\bfa,\mathbf{b}) \in
    \nn(\mathcal{A},\mathcal{B})} R_{(\bfa,\mathbf{b})} \otimes_{\kk} S_{\bfa},\]
  for each $\mathbf{b}\in\nn\mathcal{B}$ there is an $\bfa\in\nn\mathcal{A}$ such
  that $R_{(\bfa,\bfb)} \otimes_{\kk} S_{\bfa} \neq 0$.  Since $(I)_{\mathbf{b}}
  \neq \kk[x]_{\mathbf{b}}$, for each $\mathbf{b}$ there exists an $\bfa$ such
  that $R_{(\bfa,\bfb)} \neq 0$.  Now the statement holds since $S_{\bfa} \neq
  0$.
\end{proof}

\begin{ex}[Monomial primary decomposition]
  \label{ex:redundant-decomposition}
  For monomial ideals $I,J \subseteq \kk[x] = \kk[x_1, \ldots, x_n]$ with the fine
  grading on $\kk[x]$, we have $I \times_\cala J = I + J$.  This formula
  and~\eqref{eq:primary} yield a highly redundant formula for the irreducible
  decomposition of a monomial ideal:
\[
\langle x^{\bfu_1}, \cdots, x^{\bfu_r} \rangle = \bigcap_{j_1, \ldots, j_r
  \in [n] } \langle x_{j_1}^{u_{1j_1}}, \cdots, x_{j_r}^{u_{rj_r}} \rangle.
\]
For an explicit example consider:
\begin{align*}
  \langle x^2y, xy^2 \rangle & = \langle x^2y \rangle \times_\cala \langle xy^2
  \rangle = ( \langle x^2 \rangle \cap \langle y \rangle ) \times_\cala (
  \langle x \rangle \cap \langle y^2 \rangle ) = \\
  & \qquad \langle x^2, x \rangle \cap \langle x^2, y^2 \rangle \cap \langle x, y
  \rangle \cap \langle y, y^2 \rangle = \langle x \rangle \cap \langle x^2, y^2
  \rangle \cap \langle y \rangle.
\end{align*}
Redundancy arises in the decomposition as this toric fiber product does not
satisfy the conditions of Theorem~\ref{thm:irredundant}, with respect to the two
pairs of ideals $\< x \>, \<y\>$ and $\<y^2\>$, $\<x^2\>$.  Finally, the
decomposition can be redundant even when the ideals are radical, as the following
calculation illustrates:
  \begin{align*}
    \langle xy, xz \rangle & = \langle xy \rangle \times_\cala \langle  xz \rangle
    = ( \langle x \rangle \cap \langle y \rangle ) \times_\cala  ( \langle x
    \rangle \cap \langle z \rangle ) = \\
    & \qquad \langle x, x  \rangle \cap \langle x, z \rangle \cap \langle y, x \rangle \cap \langle y, z \rangle =  \langle x  \rangle \cap  \langle y, z \rangle.
  \end{align*}
\end{ex}


\section{Generators of toric fiber products of toric ideals} \label{sec:generators}

To each higher codimension toric fiber product there is a natural codimension zero
product (Definition~\ref{def:asscodimzero}) which contributes many of the
generators.  There are also additional generators \emph{glued} from certain pairs
of generators of the original ideals.  Keeping track of the different
contributions requires substantial notation which we found managable only in the
case of toric ideals.  To verify our results we require that the generating sets
of the original ideals satisfy the \emph{compatible projection property}
(Definition~\ref{def:comp-proj-prop}).  Any generating set can be extended to one
that satisfies this property, but it may be inscrutable how to do so.  In special
cases, however, the condition becomes clear.  For instance, in codimension one
toric fiber products the simpler \emph{slow-varying condition}
(Definition~\ref{def:slow-varying}) implies the compatible projection property.

Let $I \tfp J$ be any toric fiber product.  Define the ideal $\tilde{I}$ by
\[
\tilde{I} = (I(X) + \langle x^i_j - X^i_j q^i : i \in [r], j \in [s_i] \rangle)
\cap \kk[x]
\]
where $X^{i}_{j}$, and $q^{i}$ are indeterminates and $I(X)$ denotes the ideal
obtained by replacing all occurrences of $x^{i}_{j}$ with $X^{i}_{j}$.  Define
$\tilde{J}\subset\kk[y]$ in the analogous way.  Let $\tilde{\cala} = \{e_1,
\ldots, e_r \}$ be the standard unit vectors in~$\nn^r$.  By construction,
$\tilde{I}$ and $\tilde{J}$ are homogeneous with respect to the grading induced by
$\deg(x^i_j) = \deg(y^i_k) = e_i$.  Consequently $\tilde{I}$ is the subideal of
$I$ generated by all $\tilde{\cala}$-homogeneous elements.  This property could
also be used to define~$\tilde{I}$.  Hence $\tilde{I} \subseteq I$ and similarly
$\tilde{J} \subseteq J$.

\begin{defn}\label{def:asscodimzero}
  The ideal $\tilde{I} \tfpo \tilde{J}$ is the \emph{associated codimension zero
    toric fiber product} to $I \tfp J$.
\end{defn} 

In this section, $I = I_\calb$ and $J = J_\calc$ are toric ideals.  As in
Section~\ref{sec:vector}, we describe their toric fiber product and its associated
codimension zero product by their vector configurations.  Consider the linearly
independent vector configuration $\tilde{\cala} = \{ (\bfa^i, e_i) : i \in [r]
\}$, where $e_{i}$ is the $i$th basis vector of~$\zz^{r}$.  Define vector
configurations
\[
\tilde{\calb} = \{ (\bfb^i_j, e_{i}) : i \in [r], j \in [s_i] \} 
\quad \quad \mbox{ and } \quad \quad 
\tilde{\calc} = \{  (\bfc^i_k,e_i) : i \in [r], k \in [t_i]  \}.
\]
Then $\widetilde{I_\calb} = I_{\tilde{\calb}}$, $\widetilde{J_\calc} =
J_{\tilde{\calc}}$, and
\[
\widetilde{I_\calb} \tfpo \widetilde{J_\calc} = I_{\tilde{\calb} \tfpo \tilde{\calc}}.
\]

To describe generators of the toric ideal $I_{\calb \tfp \calc}$, we first relate
them to Markov bases, via the fundamental theorem~\cite{Diaconis1998}.  Let $A \in
\zz^{d \times n}$ be a matrix, which defines a toric ideal $I_A = \< p^\bfu -
p^\bfv : A \bfu = A \bfv \> \subset \kk[p_{1},\dots,p_{n}]$.  Hence, binomial
generators of $I_A$ correspond to elements in $\ker A$.  The matrix $A$ defines an
$\nn$-linear map $\nn^n \to \zz^d$ whose image is the affine semigroup~$\nn A$.
Let $\bfb \in \nn A$.  The \emph{fiber} of $\bfb$ is the set $A^{-1}[\bfb] := \{
\bfu \in \nn^n : A \bfu = \bfb \}$.  Let $\mathcal{F} \subseteq \ker A$.  For each
$\bfb \in \nn A$ we associate a graph $A^{-1}[\bfb]_\calf$, with vertex set
consisting of all lattice points in $A^{-1}[\bfb]$ and an edge between $\bfu,\bfv
\in A^{-1}[\bfb]$ if either $\bfu - \bfv$ or $\bfv - \bfu \in \calf$.  A finite
subset $\calf \subseteq \ker A$ is a \emph{Markov basis} of $A$ if the graph
$A^{-1}[\bfb]_\calf$ is connected for each $\bfb \in \nn A$.  The fundamental
theorem of Markov bases connects these lattice-based definitions with the
generators of the toric ideal~$I_A$.

\begin{thm}[Fundamental Theorem of Markov Bases~\cite{Diaconis1998}]\label{thm:Diaconis}
  A finite subset $\calf \subseteq \ker A$ is a Markov basis of $A$ if and only if
  the set of binomials $\{ p^{\bff^+} - p^{\bff^-} : \bff \in \calf \}$ 
  generates~$I_{A}$.
\end{thm} 
The fundamental theorem implies that we can describe generating sets of toric
ideals, and especially important for us, toric fiber products of toric ideals, in
terms of lattice point combinatorics.
We use tableau notation for binomials and vectors.  To explain it, let
\begin{equation*}
x^{i_1}_{j_1} x^{i_2}_{j_2}  \cdots x^{i_n}_{j_n} - x^{i_1'}_{j_1'} x^{i_2'}_{j_2'}  \cdots x^{i_n'}_{j_n'}
\end{equation*}
be a homogeneous binomial in~$\kk[x]$.  To this binomial we associate the tableau of indices:
\begin{equation*}
\left[
\begin{array}{cc}
i_1 & j_1 \\
i_2 & j_2 \\
\vdots & \vdots \\
i_n & j_n 
\end{array}
\right]
-
\left[
\begin{array}{cc}
i_1' & j_1' \\
i_2' & j_2' \\
\vdots & \vdots \\
i_n' & j_n' 
\end{array}
\right].
\end{equation*}
Similarly, we can define the tableau associated to binomials in $\kk[y]$ and
$\kk[z]$, which might look like
\begin{equation*}
\left[
\begin{array}{cc}
i_1 & k_1 \\
i_2 & k_2 \\
\vdots & \vdots \\
i_n & k_n 
\end{array}
\right]
-
\left[
\begin{array}{cc}
i_1' & k_1' \\
i_2' & k_2' \\
\vdots & \vdots \\
i_n' & k_n' 
\end{array}
\right]
\quad \quad \mbox{ and }
\quad \quad
\left[
\begin{array}{ccc}
i_1 & j_1& k_1 \\
i_2 & j_2 & k_2\\
\vdots & \vdots & \vdots  \\
i_n & j_n & k_n 
\end{array}
\right]
-
\left[
\begin{array}{ccc}
i_1' & j_1' & k_1' \\
i_2' & j_2' & k_2' \\
\vdots & \vdots & \vdots\\
i_n' & j_n'& k_n' 
\end{array}
\right]
\end{equation*}
respectively.  Tableau notation greatly simplifies the description of Markov bases
of toric fiber products.


\subsection{Codimension zero toric fiber products}\label{sec:codim0}
We review the codimension zero case from~\cite{Sullivant2007} since generators of
the associated codimension zero toric fiber product are needed in our
construction.  Let $f \in I_{\calb}$ be a binomial written in tableau notation as
\begin{equation*}
f = \left[
\begin{array}{cc}
i_1 & j_1 \\
i_2 & j_2 \\
\vdots & \vdots \\
i_n & j_n 
\end{array}
\right]
-
\left[
\begin{array}{cc}
i_1' & j_1' \\
i_2' & j_2' \\
\vdots & \vdots \\
i_n' & j_n' 
\end{array}
\right].
\end{equation*}
Since $\cala$ is linearly independent, if $f \in I_{\calb}$, then the multiset of
indices $\{i_{1}, \ldots i_{n} \}$ equals the multiset of indices $\{i'_{1},
\ldots i'_{n} \}$.  So after rearranging the rows of the tableau, we can assume
that we have the following form:
\begin{equation*}
f = \left[
\begin{array}{cc}
i_1 & j_1 \\
i_2 & j_2 \\
\vdots & \vdots \\
i_n & j_n 
\end{array}
\right]
-
\left[
\begin{array}{cc}
i_1 & j_1' \\
i_2 & j_2' \\
\vdots & \vdots \\
i_n & j_n' 
\end{array}
\right].
\end{equation*}
Let $k_{1}, \ldots, k_{n}$ be a collection of indices such that
$z^{i_{t}}_{j_{t}k_{t}}$ is a variable in $\kk[z]$ for each~$t$.  Construct the
new polynomial
\begin{equation*}
\tilde{f} = \left[
\begin{array}{ccc}
i_1 & j_1 & k_{1}\\
i_2 & j_2 & k_{2}\\
\vdots & \vdots & \vdots \\
i_n & j_n & k_{n} 
\end{array}
\right]
-
\left[
\begin{array}{ccc}
i_1 & j_1' & k_{1}\\
i_2 & j_2' & k_{2}\\
\vdots & \vdots & \vdots \\
i_n & j_n' & k_{n} 
\end{array}
\right].
\end{equation*}
For a set of binomials $\mathcal{F} \subseteq I_{\calb}$ let ${\rm Lift}(\calf)$
to be the set of all binomials $\tilde{f}$ for all $f \in \calf$ and allowable
$k_{1}, \ldots, k_{n}$.  Similarly, for a collection of binomials $\calg \subseteq
J_{\calc}$, we can define ${\rm Lift} (\calg)$.

Lastly, we introduce a set ${\rm Quad}$ which consists of all binomial quadrics of
the form
\begin{equation*}
\tilde{f} = \left[
\begin{array}{ccc}
i & j_1 & k_{1}\\
i & j_2 & k_{2}\\
\end{array}
\right]
-
\left[
\begin{array}{ccc}
i & j_1 & k_{2}\\
i & j_2 & k_{1}\\
\end{array}
\right].
\end{equation*}

\begin{thm}[Codimension zero toric fiber products,~\cite{Sullivant2007}]\label{thm:codimzero}
  Let $I_{\calb} \subseteq \kk[x]$ and $J_{\calc} \subseteq \kk[y]$ be homogeneous
  with respect to the grading by $\cala$, and suppose that $\cala$ is linearly
  independent.  Let $\calf \subseteq I_{\calb}$ and $\calg \subseteq J_{\calc}$ be
  binomial generating sets.  Then
  \begin{equation*}
    \Lift(\calf) \cup \Lift(\calg) \cup \Quad
  \end{equation*}
  is a generating set of the codimension zero toric fiber product $I_{\calb} \tfp J_{\calc}$.
\end{thm}


\subsection{The compatible projection property}
Suppose that $f \in I_\calb$ and $g \in J_\calc$ are two binomials of degree $n$,
written in tableau notation as
\begin{equation*}
f = \left[
\begin{array}{cc}
i_1 & j_1 \\
i_2 & j_2 \\
\vdots & \vdots \\
i_n & j_n 
\end{array}
\right]
-
\left[
\begin{array}{cc}
i_1' & j_1' \\
i_2' & j_2' \\
\vdots & \vdots \\
i_n' & j_n' 
\end{array}
\right]
\quad \quad
\mbox{ and }
\quad \quad
g = 
\left[
\begin{array}{cc}
i_1 & k_1 \\
i_2 & k_2 \\
\vdots & \vdots \\
i_n & k_n 
\end{array}
\right]
-
\left[
\begin{array}{cc}
i_1' & k_1' \\
i_2' & k_2' \\
\vdots & \vdots \\
i_n' & k_n' 
\end{array}
\right].
\end{equation*}
In particular assume that the first column of the leading and trailing monomial of
$f$ agrees with the first column of the leading and trailing monomial of $g$,
respectively.  In this situation, we define $\glue(f,g)$ to be the binomial
\begin{equation*}
\glue(f,g) =
\left[
\begin{array}{ccc}
i_1 & j_1& k_1 \\
i_2 & j_2 & k_2\\
\vdots & \vdots & \vdots  \\
i_n & j_n & k_n 
\end{array}
\right]
-
\left[
\begin{array}{ccc}
i_1' & j_1' & k_1' \\
i_2' & j_2' & k_2' \\
\vdots & \vdots & \vdots\\
i_n' & j_n'& k_n' 
\end{array}
\right].
\end{equation*}
Let $\kk[w] := \kk[w^{1}, \ldots, w^{r}]$, and define $\kk$-algebra homomorphisms
$\phi_{xw}$ and $\phi_{yw}$ by
\begin{equation*}
\phi_{xw}:  \kk[x] \rightarrow \kk[w]  \quad x^{i}_{j} \mapsto  w^{i},
\end{equation*}
\begin{equation*}
\phi_{yw}:  \kk[y] \rightarrow \kk[w]  \quad y^{i}_{k} \mapsto  w^{i}.
\end{equation*}
In general, we define the gluing operation on pairs of binomials $f \in I_\calb$
and $g \in J_\calc$ such that $\phi_{xw}(f) = w^{\bfv_1}(w^{\bfu_1} - w^{\bfu_2})$
and $\phi_{yw}(g) = w^{\bfv_2}(w^{\bfu_1} - w^{\bfu_2})$.  The binomial part in
both products are assumed to be the same, and we say that $f$ and $g$ are
\emph{compatible}.  Furthermore, we can assume that $\gcd(w^{\bfv_1}, w^{\bfv_2})
= 1$, by not factoring the polynomials completely.

Define $L(w^{\bfv_2})$ to be the set of all monomials $x^\bfv$ in $\kk[x]$ such
that $\phi_{xw}(x^\bfv) = w^{\bfv_2}$.  Similarly, define $R(w^{\bfv_1})$ to be
the set of monomials $y^\bfv$ in $\kk[y]$ such that $\phi_{yw}(y^\bfv) =
w^{\bfv_1}$.  By construction if $x^{\bfv} \in L(w^{\bfv_2})$ and $y^{\bfv'} \in
R(w^{\bfv_1})$ then $x^\bfv f$ and $y^{\bfv'} g$, when written as tableau and
after reordering rows, have exactly the same first column.  Thus, we can form the
binomial $\glue(x^\bfv f, y^{\bfv'} g)$.

\begin{defn}
  Let $\calf \subseteq I_\calb$ and $\calg \subseteq J_\calc$ consist of
  binomials.  The \emph{glued binomials} are
  \begin{equation*} \Glue(\calf, \calg) = \{ \glue(x^\bfv f, y^{\bfv'}
    g) : f \in \calf, g \in \calg \mbox{ compatible, } x^\bfv \in
    L(w^{\bfv_2}), y^{\bfv'} \in R(w^{\bfv_1}) \}.
\end{equation*}
The set of exponent vectors of binomials in $\Glue (\calf, \calg)$ is $\GGlue
(\calf, \calg)$.
\end{defn}

\begin{prop}
  If $\calf \subseteq I_\calb$ and $\calg \subseteq J_\calc$ are sets of binomials
  then
  \[{\Glue}(\calf, \calg)  \subset I_{\calb \tfp \calc}.\]
\end{prop}

\begin{proof}
  For toric ideals, a binomial $h \in \kk[z]$ belongs to $I_{\calb \tfp
    \calc}$ if and only if $\phi_{zx}(h) \in I_\calb$ and $\phi_{zy}(h) \in
  I_\calc$, where $\phi_{zx}$ and $\phi_{zy}$ are the $\kk$-algebra homomorphisms
  \begin{gather*}
    \phi_{zx} : \kk[z] \rightarrow \kk[x], \quad z^{i}_{j,k} \mapsto x^{i}_{j},\\
    \phi_{zy} : \kk[z] \rightarrow \kk[y], \quad z^{i}_{j,k} \mapsto y^{i}_{k}.
  \end{gather*}
  For any $\glue(x^\bfv f, y^{\bfv'} g)$ where $f \in I_{\calb}$ and $g \in
  J_{\calc}$, we have $\phi_{zx}( \glue(x^\bfv f, y^{\bfv'} g) ) = x^\bfv f
  \in I_{\calb}$, and $\phi_{zy}( \glue(x^\bfv f, y^{\bfv'} g) ) = y^{\bfv'}
  g \in J_{\calc}$.
\end{proof}

Consider the natural $\nn$-linear projection maps $\gamma : \nn^{\calb \tfp \calc}
\to \nn^r, \gamma(e^i_{jk}) = e_i$, $\gamma_1 : \nn^\calb \rightarrow \nn^r,
\gamma_1( e^i_j) = e_i$, and $\gamma_2: \nn^\calc \rightarrow \nn^r,
\gamma_2(e^i_k) = e_i$.
These projections evaluate the additional multidegrees appearing in the definition
of the associated codimension zero product.  They are also defined on the fibers
$\calb^{-1}[\bfb]$ and $\calc^{-1}[\bfc]$ and the graphs $\calb^{-1}[\bfb]_\calf$
and $\calc^{-1}[\bfc]_\calg$.  Note that if $\bff \in \ker \calb$ then $\gamma_1(\bff) \in \ker \cala$, and similarly for $\gamma$, and $\gamma_2$.

\begin{defn}
  Let $\calf \subseteq \ker \calb$.  The graph $\gamma_1( \calb^{-1}[\bfb]_\calf)$
  has vertex set $\gamma_1(\calb^{-1}[\bfb])$ and an edge between $\bfu'$ and
  $\bfv'$ if there are $\bfu, \bfv \in \calb^{-1}[\bfb]$ such that $\bfu$ and
  $\bfv$ are connected by an edge in $\calb^{-1}[\bfb]_\calf$ and $\gamma_1(\bfu)
  = \bfu'$ and $\gamma_1(\bfv) = \bfv'$.  Similarly define the graphs $\gamma_2(
  \calc^{-1}[\bfc]_\calg)$ and $\gamma((\calb \tfp \calc)^{-1}[
  (\bfb,\bfc)]_\calh)$ where $\calg \subseteq \ker \calc$ and $\calh \subseteq
  \ker \calb \tfp \calc$.  These are the \emph{projection graphs}.
\end{defn}
Given two graphs $G$ and $H$ with overlapping vertex sets, their intersection $G
\cap H$ is the graph with vertex set $V(G) \cap V(H)$ and edge set $E(G) \cap
E(H)$.

\begin{defn}\label{def:comp-proj-prop}
  Let $\calf \subseteq \ker \calb$ and $\calg \subseteq \ker \calc$.  The pair
  $\calf$ and $\calg$ has the \emph{compatible projection property} if for all
  $\bfb \in \nn\calb$ and $\bfc \in \nn\calc$ such that $\pi_1(\bfb) =
  \pi_2(\bfc)$, the graph
  \begin{equation*}
    \gamma_1( \calb^{-1}[\bfb]_\calf) \cap \gamma_2( \calc^{-1}[\bfc]_\calg)
  \end{equation*}
  is connected.
\end{defn}
The next lemma is the main technical result allowing us to produce generating sets
for toric fiber products.

\begin{lemma}\label{lem:gluegraph}
  Let $\calf \subseteq \ker \calb$ and $\calg \subseteq \ker \calc$.  Let $\bfb
  \in \nn\calb$ and $\bfc \in \nn\calc$ such that $\pi_1(\bfb) = \pi_2(\bfc)$.
  Then
  \begin{equation*}
    \gamma( (\calb \tfp \calc)^{-1}[ (\bfb, \bfc)]_{{\GGlue}(\calf,\calg)}) \quad = \quad \gamma_1( \calb^{-1}[\bfb]_\calf)  \cap \gamma_2( \calc^{-1}[\bfc]_\calg).
\end{equation*}
\end{lemma}

\begin{proof}
We must show:
\begin{enumerate}
\item $V(\gamma( (\calb \tfp \calc)^{-1}[ (\bfb, \bfc)]_{{\GGlue}(\calf,\calg)})) \quad = \quad V(\gamma_1( \calb^{-1}[\bfb]_\calf) )
  \cap V(\gamma_2( \calc^{-1}[\bfc]_\calg)),$
\item $E(\gamma( (\calb \tfp \calc)^{-1}[ (\bfb, \bfc)]_{{\GGlue}(\calf,\calg)})) \quad = \quad E( \gamma_1( \calb^{-1}[\bfb]_\calf) )
  \cap E(\gamma_2( \calc^{-1}[\bfc]_\calg))$.
\end{enumerate}
In both part (1) and (2) the containment ``$ \subseteq$'' is straightforward, by
projecting.  Indeed, if $\bfu \in (\calb \tfp \calc)^{-1}[(\bfb, \bfc)]$, then
applying the canonical map $\pi_{zx} : \zz^{\calb \tfp \calc} \rightarrow
\zz^{\calb}$ gives $\pi_{zx}(\bfu) \in \calb^{-1}[\bfb]$ and $\gamma(\bfu) =
\gamma_{1}(\pi_{zx}(\bfu))$.  Similarly, $\gamma(\bfu) =
\gamma_{2}(\pi_{zy}(\bfu))$.  Furthermore, if $\bfu$ and $\bfu'$ are connected by
an edge corresponding to the binomial
 $\glue (x^{\bfv}f, y^{\bfv'}g) \in \glue(\calf,
\calg)$ then $\pi_{zx}(\bfu)$ and $\pi_{zx}(\bfu')$ are connected by~$\bff$, and
$\pi_{zy}(\bfu)$ and $\pi_{zy}(\bfu')$ are connected by~$\bfg$, where
 $f = x^{\bff^+} - x^{\bff^-}$ and $g = y^{\bfg^+} - y^{\bfg^-}$.

\paragraph{\bf Proof of part (1).}
We must show that if $\bfd$ is in both $\gamma_1( \calb^{-1}[\bfb]_\calf)$ and
$\gamma_2( \calc^{-1}[\bfc]_\calg)$ then $\bfd \in \gamma( (\calb \tfp
\calc)^{-1}[ (\bfb, \bfc)]_{{\GGlue}(\calf,\calg)}])$.  By assumption there are
$\bfu_1 \in \calb^{-1}[\bfb]$ and $\bfu_2 \in \calc^{-1}[\bfc]$ such that
$\gamma_1(\bfu_1) = \gamma_2(\bfu_2) = \bfd$.  Since $\pi_1(\bfb) = \pi_2(\bfc)$
and $\gamma_1(\bfu_1) = \gamma_2(\bfu_2)$ the corresponding monomials $x^{\bfu_1}$
and $y^{\bfu_2}$ have the same $\tilde{\cala}$ degree.  Since $\tilde{\cala}$ is
linearly independent, the monomial $x^{\bfu_1} y^{\bfu_2} \in \kk[x] \otimes_\kk
\kk[y]$ is in the image of $\phi_{I_{\tilde{\calb}}, J_{\tilde{\calc}}}$.  Let
$z^\bfu$ be a monomial such that $(\tilde{\calb} \tfpo \tilde{\calc})\bfu = (\bfb,
\bfc, \bfd)$ and hence $(\calb \tfp \calc)\bfu = (\bfb, \bfc)$.  But this implies
$\bfd \in \gamma( (\calb \tfp \calc)^{-1}[ (\bfb, \bfc)]_{{\GGlue}(\calf,\calg)}])$.

\paragraph{\bf Proof of part (2).}
Suppose that $\bfd$ and $\bfe$ are both in $\gamma_1( \calb^{-1}[\bfb]_\calf) $
and $\gamma_2( \calc^{-1}[\bfc]_\calg)$, and they are connected by an edge.  We
must show that $\bfd$ and $\bfe$ are connected by an edge in $\gamma( (\calb \tfp
\calc)^{-1}[ (\bfb, \bfc)]_{{\GGlue}(\calf,\calg)}])$.  To do this, we must show
that there are $\bfw_1$ and $\bfw_2 \in (\calb \tfp \calc)^{-1}[(\bfb, \bfc)]$,
with $\gamma(\bfw_{1}) = \bfd$ and $\gamma(\bfw_{2}) = \bfe$ such that $\bfw_1 -
\bfw_2 \in {\GGlue}(\calf, \calg)$.

Since there is an edge in $\gamma_1( \calb^{-1}[\bfb]_\calf)$ between $\bfd$ and
$\bfe$, there exist $\bfu_1$ and $\bfu_2$ in $\calb^{-1}[\bfb]$ such that
$\gamma_1(\bfu_1) = \bfd$, $\gamma_1(\bfu_2) = \bfe$ and $\bfu_1 - \bfu_2 = \bff \in
\calf$.  Similarly, there are $\bfv_1$ and $\bfv_2 \in \calc^{-1}[\bfc]$ such that
$\gamma_2(\bfv_1) = \bfd$, $\gamma_2(\bfv_2) = \bfe$ and $\bfv_1 - \bfv_2 = \bfg
\in \calg$.  By part (1), there exists $\bfw_1 \in(\calb \tfp \calc)^{-1}[(\bfb,
\bfc)] $ which projects to $(\bfu_1, \bfv_1)$ and $\bfw_2 \in (\calb \tfp
\calc)^{-1}[(\bfb, \bfc)]$ which projects to $(\bfu_2, \bfv_2)$.  There are many
choices for $\bfw_1$ and $\bfw_2$.  We claim that we can choose them so that
$\bfw_1 - \bfw_2 \in {\GGlue}(\calf, \calg)$, which completes the proof.

To prove the claim, we explicitly construct these elements.  This requires an
understanding of the precise forms that $\bfu_{1}, \bfu_{2}, \bfv_{1}$, and
$\bfv_{2}$ take.  Writing $\bfu_{1} - \bfu_{2}$ and $\bfv_{1}- \bfv_{2}$ as
tableaux in block form we have:
\begin{equation*}
\bfu_{1} - \bfu_{2} = 
\begin{bmatrix}
I_{1} & J_{1}  \\
\hline
I_{2} & J_{2}  \\
\hline
I_{3} & J_{3}  
\end{bmatrix} \quad - \quad
\begin{bmatrix}
I_{1}' & J_{1}'  \\
\hline
I_{2} & J_2' \\
\hline
I_{3} & J_{3}
\end{bmatrix} 
\end{equation*}
\begin{equation*}
\bfv_{1} - \bfv_{2} = 
\begin{bmatrix}
I_{1} & K_{1}  \\
\hline
I_{2}^{*} & K_{2}  \\
\hline
I_{3}^{*} & K_{3}  
\end{bmatrix} \quad - \quad
\begin{bmatrix}
I_{1}' & K_{1}'  \\
\hline
I_{2}^{*} & K_2' \\
\hline
I_{3}^{*} & K_{3}
\end{bmatrix}. 
\end{equation*}
Note that $I,J,K$ are multisets here, not ideals.
The first two blocks of rows in the tableaux for $\bfu_{1} - \bfu_{2}$ give the
support of this difference.  This corresponds to the binomial~$f$.  The last block
of rows corresponds to the part where the vectors agree, and hence is the same in
both $\bfu_{1}$ and $\bfu_{2}$.  Similarly, the first two blocks of rows in the
tableaux for $\bfv_{1} - \bfv_{2}$ give the support of this difference.  This
corresponds to the binomial~$g$.  The last block of rows corresponds to the part
where the vectors agree, and hence is the same in both $\bfv_{1}$ and $\bfv_{2}$.

The first block of rows in both $\bfu_{1} - \bfu_{2}$ and $\bfv_{1} - \bfv_{2}$,
have the same $I_{1}$ and $I_{1}'$ because these blocks correspond to the common
binomial $(w^{\bfs_1} - w^{\bfs_2})$ in $\phi_{xw}(f) = w^{\bfr_1}(w^{\bfs_1} -
w^{\bfs_2})$ and $\phi_{yw}(g) = w^{\bfr_2}(w^{\bfs_1} - w^{\bfs_2})$.  Note that
this corresponds to $\bfd - \bfe$.  This implies that in the second and third
blocks of rows of $\bfu_{1}$ and of $\bfu_{2}$ we have exactly the same multisets
of indices in the first column.  This explains why $I_{2}$ and $I_{3}$ appear in
both the $\bfu_{1}$ and the $\bfu_{2}$ tableaux.  A similar argument shows that
$I_{2}^{*}$ and $I_{3}^{*}$ should appear in both $\bfv_{1}$ and $\bfv_{2}$.
Finally, we must have that the multiset of indices appear in $I_{2}$ and $I_{3}$
together equals the multiset of indices that appear in $I_{2}^{*}$ and $I_{3}^{*}$
together.  By our usual assumption that $\gcd( w^{\bfr_1}, w^{\bfr_2}) = 1$, we
see that the multisets $I_{2}$ and $I_{2}^{*}$ are disjoint.  This implies that,
as multisets, $I_{2} \subseteq I_{3}^{*}$ and $I_{2}^{*} \subseteq I_{3}$ .

With all this information on the structure of the tableau, we can build our
element of ${\GGlue}(\calf, \calg)$.  Indeed, we construct this binomial by
constructing its tableau form, which is:
\begin{equation*}
h = 
\begin{bmatrix}
I_{1} & J_{1} & K_{1}  \\
\hline
I_{2} & J_{2} & M \\
\hline
I_{2}^{*} & N & K_{2}  
\end{bmatrix} \quad - \quad
\begin{bmatrix}
I_{1}' & J_{1}' & K_{1}'  \\
\hline
I_{2} & J_{2}' & M \\
\hline
I_{2}^{*}& N & K_{2}'
\end{bmatrix}. 
\end{equation*}
Here $M$ is chosen so that the rows of $[I_{2} \, \, \, M]$ are a multi-subset of
the rows of $[I_{3}^{*} \, \, \, K_{3}]$, and $N$ is chosen so that the rows of
$[I_{2}^{*} \, \, \, N]$ are a multi-subset of the rows of $[I_{3} \, \, \,
J_{3}]$.  By construction $h \in {\GGlue}(\calf, \calg)$ since the $x$ monomial
corresponding to $[I_{2}^{*} \, \, \, N]$ belongs to $L( w^{\bfr_{2}})$ and the
$y$ monomial corresponding to $[I_{2} \, \, \, M]$ belongs to $R(w^{\bfr_{1}})$.

We do not yet have $\bfw_{1}$ and $\bfw_{2}$, since there might be leftover
indices from the last blocks of rows of $\bfu_{1} - \bfu_{2}$ and $\bfu_{1} -
\bfu_{2}$.  Call these remaining rows: $[I \,\, \, J] - [I \, \, \, J]$ in the
first case, and $[I \,\, \, K] - [I \, \, \, K]$ in the second.  Note that we have
the same multiset of indices $I$ in both, since we have extracted $I_{2}$ and
$I_{2}^{*}$ from both the pair $I_{2}$ and $I_{3}$ and the pair $I_{2}^{*}$ and
$I_{3}^{*}$, which had the same multiset of indices.  This means, finally, that we
have $\bfw_{1}$ and $\bfw_{2}$ in tableau notation as:
\begin{equation*}
\begin{bmatrix}
I_{1} & J_{1} & K_{1}  \\
\hline
I_{2} & J_{2} & M \\
\hline
I_{2}^{*} & N & K_{2} \\
\hline
I & J & K  
\end{bmatrix} \quad - \quad
\begin{bmatrix}
I_{1}' & J_{1}' & K_{1}'  \\
\hline
I_{2} & J_{2}' & M \\
\hline
I_{2}^{*}& N & K_{2}'  \\
\hline
I & J & K
\end{bmatrix}.
\end{equation*}
Since $\pi_{zx}(\bfw_{1}) = \bfu_{1}$ and $\pi_{zy}(\bfw_{1}) = \bfv_{1}$, this
implies $\gamma(\bfw_{1}) = \bfd$.  Similarly, $\gamma(\bfw_{2}) = \bfe$.
Finally, by construction $\bfw_{1}$ and $\bfw_{2}$ are connected by the move $h$,
which is in ${\GGlue}(\calf, \calg)$.  This completes the proof since now
\begin{equation*}
  E(\gamma( \calb \tfp \calc^{-1}[ (\bfb, \bfc)]_{{\GGlue}(\calf,\calg)}]))
  \quad \supseteq \quad E( \gamma_1( \calb^{-1}[\bfb]_\calf) ) \cap E(\gamma_2(
  \calc^{-1}[\bfc]_\calg)). \qedhere
\end{equation*}
\end{proof}

The idea of the proof of Theorem~\ref{thm:cpp} is summarized by
Figure~\ref{fig:compatproj}.  We wish to show that the graph of each fiber is
connected.  To do so we decompose the lattice $\ker \calb \times_{\cala} \calc$
into two directions.  The first direction (vertical in the figure) corresponds to
the lattice of the associated codimension zero toric fiber product.  The subgraphs
of fiber elements constrained to lie in a translate of that lattice are connected
since we have a Markov basis for the associated zero toric fiber product.  The
remaining lattice directions (essentially horizontal in the figure) arise because
the product is not actually of codimension zero.  By projecting via $\gamma$ and
showing that the image graph is connected (using Lemma~\ref{lem:gluegraph}), we
deduce that the entire graph is connected.

\begin{figure}[ht]
  \centering
  \resizebox{6.3cm}{!}{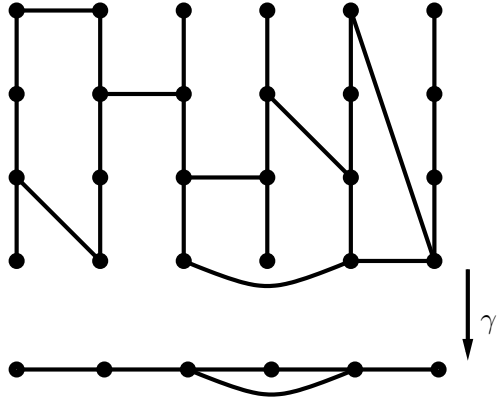}
  \caption{Illustration of connected fibers.}
  \label{fig:compatproj}
\end{figure}

\begin{thm}\label{thm:cpp}
  Let $\calh \subset \ker \tilde{\calb}\tfpo \tilde{\calc}$ be a Markov basis for
  the associated codimension zero toric fiber product.  Let $\calf \subseteq \ker
  \calb$ and $\calg \subseteq \ker \calc$.  Then $\calh \cup {\GGlue}(\calf,
  \calg)$ is a Markov basis for $ \calb \tfp \calc$ if and only if $\calf$ and
  $\calg$ have the compatible projection property.
\end{thm}

\begin{proof}
  We must show that for any $(\bfb, \bfc) \in \nn( \calb \tfp \calc)$ the graph
  $(\calb \tfp \calc)^{-1}[ (\bfb, \bfc)]_{ \calh \cup {\GGlue}(\calf, \calg)}$
  is connected.  For each $\bfd \in \nn \cald$ consider the subgraph of $ (\calb
  \tfp \calc)^{-1}[ (\bfb, \bfc)]_{ \calh \cup {\GGlue}(\calf, \calg)}$ whose
  vertices consist of all $(\bfu, \bfv) \in (\calb \tfp \calc)^{-1}[ (\bfb,
  \bfc)]$ such that $\gamma_1(\bfu) = \gamma_2(\bfv) = \bfd$.  This is precisely
  the set $ \tilde{\calb} \tfpo \tilde{\calc}^{-1}[ (\bfb, \bfc, \bfd)]$.  This
  subgraph is connected since $\calh$ is a Markov basis for $ \tilde{\calb} \tfpo
  \tilde{\calc}$.  The graph $\gamma( (\calb \tfp \calc)^{-1}[ (\bfb, \bfc)]_{
    \calh \cup \GGlue(\calf, \calg)})$ equals the graph $\gamma( (\calb \tfp
  \calc)^{-1}[ (\bfb, \bfc)]_{\GGlue(\calf, \calg)})$ because $\calh$ is contained
  in the kernel of the projection~$\gamma$.  This graph is connected since $\calf$
  and $\calg$ have the compatible projection property and by
  Lemma~\ref{lem:gluegraph}.  But if the image of a map of graphs is connected and
  each fiber is connected, then the graph itself is connected, which completes the
  proof of the if direction.
  
  Conversely, if every fiber is connected, the graph $(\calb \tfp \calc)^{-1}[
  (\bfb, \bfc)]_{ \calh \cup {\GGlue}(\calf, \calg)}$ is connected, so the graph
  $\gamma( (\calb \tfp \calc)^{-1}[ (\bfb, \bfc)]_{\GGlue(\calf, \calg)})$ is
  connected.  By Lemma~\ref{lem:gluegraph}, this equals $\gamma_1(
  \calb^{-1}[\bfb]_\calf) \cap \gamma_2( \calc^{-1}[\bfc]_\calg)$ so that $\calf$
  and $\calg$ have the compatible projection property.
\end{proof}

Theorem~\ref{thm:cpp} gives an explicit way to construct a Markov basis for $\calb
\tfp \calc$.  However, there remains a serious difficulty in finding sets $\calf
\subset \ker \calb$ and $\calg \subset \ker \calc$ which have the compatible
projection property.  In general, it is not true that $\calf$ and $\calg$ can
be arbitrary Markov bases of $\calb$ and $\calc$.


\subsection{Slow-varying Markov bases}

In the remainder of the section, we describe the \emph{slow-varying} condition
(generalizing \cite{Engstrom2008}) which, if the codimension is one, can be used
to show that a given pair of Markov bases satisfies the compatible projection
property.

\begin{defn}\label{def:slow-varying}
  Suppose that $\calb \tfp \calc$ is a codimension one toric fiber product.  Let
  $\bfh \in \zz^{r}$ be non-zero.  Let $\calf \subseteq \ker
  \calb$ and $\calg \subseteq \ker \calc$.  Then $\calf$ and $\calg$ are
  \emph{slow-varying} with respect to $\bfh$ if for all $\bff \in \calf$,
  $\gamma_1(\bff) = 0$, or $\pm \bfh$; and for all $\bfg \in \calg$,
  $\gamma_2(\bfg) = 0$ or $\pm \bfh$.
\end{defn}

\begin{prop}\label{prop:slowvaryingcheck}
  Let $\bfh$ generate $\ker \cala$.  If the maximum $1$-norm of any element in
  $\calf$ or $\calg$ is less than $2 \| \bfh\|_{1}$, then $\calf$ and $\calg$ are
  slow-varying with respect to $\bfh$.
\end{prop}

\begin{proof}
 Since $\gamma_{1}(\bff)$ must be a multiple of $\bfh$ and $\|\gamma_1(\bff) \|_1 \leq \| \bff \|_1$ , if
$\| \bff \|_{1} < 2 \| \bfh \|_{1}$ then $\gamma_{1}(\bff)$ is either
$0$ or $\pm \bfh$.  A similar statement holds for $\gamma_{2}(\bfg)$.
\end{proof}

\begin{thm}\label{thm:slowvarying}
  Suppose that $\calb \tfp \calc$ is a codimension one toric fiber product.  Let
  $\calh$ be a Markov basis for $\tilde{\calb} \tfpo \tilde{\calc}$.  Let $\calf$
  and $\calg$ be Markov bases for $\calb$ and $\calc$ that are slow-varying with
  respect to $\bfh \in \ker \cala$.  Then $\calh \cup \GGlue(\calf, \calg)$ is a
  Markov basis for $\calb \tfp \calc$.
\end{thm}

\begin{proof}
  Since the toric fiber product is codimension one, the vertex sets of  the graphs 
  $\gamma_1(\calb^{-1}[\bfb]_\calf)$ and $\gamma_2( \calc^{-1}[\bfc]_\calg)$ are
  subsets of the lattice
  $\zz \bfh$.  Since $\calf$ and $\calg$ are Markov bases, these graphs are
  connected.  By the slow-varying condition, the edges connect two points whose
  difference is $\pm \bfh$.  Hence the graphs $\gamma_1( \calb^{-1}[\bfb]_\calf)$
  and $\gamma_2( \calc^{-1}[\bfc]_\calg)$ are intervals of ordered points.  The
  intersection of two such graphs is another graph of the same type, and is also
  connected.  Thus $\calf$ and $\calg$ have the compatible projection property and
  Theorem~\ref{thm:cpp} then implies that $\calh \cup {\GGlue}(\calf, \calg)$ is
  a Markov basis for $\calb \tfp \calc$.
\end{proof}

In general, we cannot expect to simply use minimal Markov bases $\calf$ and
$\calg$ of $\calb$ and $\calc$ to construct a Markov basis of $\calb \tfp \calc$.
Indeed, even in the codimension one case when those Markov bases are not
slow-varying, we might have the situation that every $f \in \calf$
satisfies $\gamma_{1}(f) = 0, \pm \bfh, \pm 2 \bfh$ and 
every $g \in \calg$ satisfies
$\gamma_{2}(g) = 0, \pm \bfh, \pm 2 \bfh$, but there are elements $h$ in the
Markov basis for $\calb \tfp \calc$, with $\gamma(h) = m \bfh$ for $m$ large.  The
problem is illustrated by Figure~\ref{fig:noslow}, which would require
augmenting the sets $\calf$ and $\calg$ with some elements
that had $\gamma_{1}(f) = \gamma_{2}(g) = \pm3 \bfh$ to guarantee
the compatible projections property.
\begin{figure}[ht]
  \centering
  \includegraphics[width=10cm]{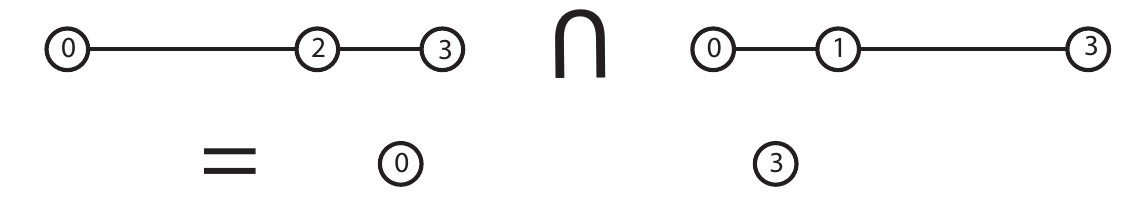}
  \caption{A codimension one toric fiber product that is not slow-varying.}
  \label{fig:noslow}
\end{figure}


\section{Application: Markov bases for hierarchical models}\label{sec:markov}

Let $\Gamma$ be a simplicial complex with vertex set $V$, and let $d \in
\zz^{V}_{\geq 2}$ a vector of integers.  These data define a hierarchical model as
in Section~\ref{sec:2:hierarchical}, and hence a toric ideal~$I_{\Gamma,d}$.  For
any homogeneous ideal $I$, let $\mu(I)$ denote the largest degree of a minimal
generator of $I$, which is an invariant of the ideal.  This is a coarse measure of
the complexity of the ideal~$I$.  If $\Gamma$ is a graph and $d_{v} = 2$ for all
$v \in V$, $\mu(I_{\Gamma,d})$, is an invariant of $\Gamma$ dubbed the
\emph{Markov width} in~\cite{Develin2003}.  We calculate $\mu(\Gamma,d) :=
\mu(I_{\Gamma,d})$ for certain simplicial complexes $\Gamma$ and vectors~$d$.  The
results of Section~\ref{sec:generators} are also useful to explicitly construct
Markov bases of these hierarchical models.

The ideal $I_{\Gamma,d}$ is the toric ideal of a matrix $A_{\Gamma,d}$ with
columns indexed by elements $i \in \D_{V}$.  Each column is given by the formula
$$
(A_{\Gamma,d})_{i} =  \bigoplus_{F \in {\rm facet}(\Gamma)} e^{F}_{i_{F}}  \in \bigoplus_{F \in {\rm facet}(\Gamma)}  \zz^{\D_{F}}
$$
where $\{e^{F}_{j_{F}} : j_{F} \in \D_{F} \}$ is the standard basis for
$\zz^{\D_{F}}$.  For $S \subset V$, let $\Gamma_{S}$ denote the induced subcomplex
on $S$ (that is, $\Gamma_{S} = \{ F \in \Gamma : F \subseteq S \}$).  The matrix
$A_{\Gamma_{S}, d_{S}}$ induces a grading on $I_{\Gamma, d}$ by $\deg(p_{i}) =
(A_{\Gamma_{S}, d_{S}})_{i_{S}}$.  This grading is the \emph{$S$-grading}.

\begin{prop}\label{prop:gammatfp}
  Let $\Gamma$ be a simplicial complex with $\Gamma = \Gamma_{1} \cup \Gamma_{2}$,
  where the vertex set of $\Gamma_{1}$ and $\Gamma_{2}$ are $V_{1}$ and $V_{2}$,
  respectively.  Let $S = V_{1} \cap V_{2}$ and suppose that $\Gamma_{1} \cap
  \Gamma_{2} = (\Gamma_{1})_{S} = (\Gamma_{2})_{S}$.  Then
\begin{equation*}
I_{\Gamma,d}  =  I_{\Gamma_1,d_{V_{1}}}  \times_{A_{\Gamma_{S},d_{S}}}  I_{\Gamma_2,d_{V_{2}}}.
\end{equation*}
\end{prop}

\begin{proof}
  Since all the ideals are toric, it suffices to show that the fiber product of
  the vector configurations $A_{\Gamma_{1},d_{V_{1}}}, A_{\Gamma_{2},d_{V_{2}}}$
  equals~$A_{\Gamma,d}$.
  For $i \in \D_{V}$ let $\bfb_{i_{V_{1}}}$ be the column of $A_{\Gamma_{1},
    d_{V_{1}}}$ indexed by $i_{V_{1}}$.  Similarly, define $\bfc_{i_{V_{2}}}$, and
  $\bfa_{i_{S}}$ as the appropriate columns of $A_{\Gamma_{2}, d_{V_{2}}}$ and
  $A_{\Gamma_{S}, d_{S}}$, respectively.  For $l=1,2$, let $\pi_l :
  \rr^{\D_{V_{l}}} \rightarrow \rr^{\D_{S}}$ be the linear projections induced by
  the grading that gives $\deg p_{i_{V_{1}}} = \deg p_{i_{V_{2}}} = \bfa_{i_{S}}$.
  The toric fiber product of vector configurations is
$$
A_{\Gamma_{1}, d_{V_{1}}} \times_{A_{\Gamma_{S}, d_{S}}} A_{\Gamma_{2}, d_{V_{2}}}  =
\{  (\bfb_{j} , \bfc_{k} :  j \in d_{V_{1}}, k \in d_{V_{2}}, j_{S} = k_{S} \}.
$$
This means that columns of the toric fiber product vector configuration have the form 
\[
\bigg( \bigoplus_{F \in {\rm facet}(\Gamma_{1})} \, \,  e^{F}_{i_{F}} \bigg) 
\oplus 
\bigg( \bigoplus_{F \in {\rm facet}(\Gamma_{2})} \, \,  e^{F}_{i_{F}}\bigg).
\]
If a facet $F$ appears in both $\Gamma_{1}$ and $\Gamma_{2}$, we can delete one of
the corresponding collections of rows of $A_{\Gamma_{1}, d_{V_{1}}}
\times_{A_{\Gamma_{S}, d_{S}}} A_{\Gamma_{2}, d_{V_{2}}} $, without changing the
kernel of the matrix, and hence the toric ideal.  After eliminating repeats, we
precisely have the matrix of $A_{\Gamma,d}$. 
\end{proof}

In~\cite{Hosten2002}, the codimension of a hierarchical model ($\Gamma$,$d$) is
given by the formula
\begin{equation}\label{eq:dimformula}
\sum_{F \notin \Gamma} \prod_{i \in F} (d_{i} -1).
\end{equation}
Hence, the toric fiber product from Proposition~\ref{prop:gammatfp} is a
codimension zero toric fiber product if and only if $\Gamma_{S} = 2^{S}$, and is a
codimension one toric fiber product if and only if $\Gamma_{S} = 2^{S} \setminus
\{S\}$ and $d_{s} = 2$ for all $s \in S$.  

\begin{prop}\label{prop:gammaassoc0}
Let $\Gamma$ be a simplicial complex with $\Gamma = \Gamma_{1} \cup \Gamma_{2}$, where the vertex set of $\Gamma_{1}$ and $\Gamma_{2}$ are $V_{1}$ and $V_{2}$, respectively, and  $\Gamma_{1} \cap \Gamma_{2}  =  (\Gamma_{1})_{S} = (\Gamma_{2})_{S}$.  Let $\tilde{\Gamma} = \Gamma \cup 2^{[S]}$, and similarly $\tilde{\Gamma}_{1} = \Gamma_{1} \cup 2^{[S]}$ and $\tilde{\Gamma}_{2} = \Gamma_{2} \cup 2^{[S]}$.  Then
$$
\tilde{I}_{\Gamma,d} = I_{\tilde{\Gamma},d} =  I_{\tilde{\Gamma}_{1}, d_{V_{1}}} \times_{A_{2^{[S]},d_{S}}} I_{\tilde{\Gamma}_{2}, d_{V_{2}}}.
$$
\end{prop}

\begin{proof}
  It suffices to show that for $I_{\Gamma, d}$, the construction of
  $\tilde{I}_{\Gamma,d}$ gives $I_{\tilde{\Gamma},d}$.  Since $I_{\Gamma, d}$ is
  the kernel of a ring homomorphism, the construction of $$\tilde {I}_{\Gamma,d} =
  (I_{\gamma, d}(P) + \langle p_{i} - P_{i} q_{i_{S}} : i \in d \rangle) \cap
  \kk[p]$$ simply modifies that parametrization by taking $\tilde{\phi}(p_{i}) =
  \phi(p_{i}) q_{i_{S}}$.  Thus, we have that $\tilde{I}_{\Gamma,d} = \ker
  \tilde{\phi}$ where
  \begin{equation*}
    \tilde{\phi}(p_{i})  =  q_{i_{S}}  \prod_{F \in {\rm facet}(\Gamma)}  a^{F}_{i_{F}}.
  \end{equation*}
  We can delete all the $a^{F}_{i_{F}}$ parameters when $F \subseteq S$, since
  this does not change the kernel of the homomorphism.  But then, this is
  precisely the parameterization associated with~$\tilde{\Gamma}$.
\end{proof}


\subsection{Small examples}
In this section we restrict to the binary case where $d_{v} = 2 $ for all $v \in
V$.  To this end, let $\bf2$ be the vector whose every coordinate is a~$2$.  We
illustrate the constructions Quad, Lift, and Glue for small hierarchical models.
Tableaux for binomials in hierarchical models have one column for each $i\in V$
(and as always one row per variable appearing in a monomial). For example, we
represent the binomial $p_{111}p_{122}p_{212}p_{221} - p_{112} p_{121} p_{211}
p_{222}$ as the tableau:
$$
\begin{bmatrix}
1 & 1 & 1 \\
1 & 2 & 2 \\
2 & 1 & 2 \\
2 & 2 & 1
\end{bmatrix}
- 
\begin{bmatrix}
1 & 1 & 2 \\
1 & 2 & 1 \\
2 & 1 & 1 \\
2 & 2 & 2
\end{bmatrix}.
$$

\begin{lemma}\label{lem:cones}$ $\\[-3ex]
\begin{enumerate}
\item Let $S_{V} = 2^{V} \setminus \{V\}$, be the boundary of a $(\#V
  -1)$-dimensional simplex.  Then 
  $I_{S_{V}, {\bf 2}}$ is generated by a single binomial:
\[
\prod_{i  \in \D_{V}: \| i \|_{1} {\rm even}} p_{i} - \prod_{i \in \D_{V}: \| i \|_{1} {\rm odd}} p_{i}. 
\]
\item Let $\Gamma$ be a simplicial complex on $V$, let
$${\rm cone}_{v}(\Gamma)  =  \Gamma \cup  \{F \cup \{v\} :  F \in \Gamma \}$$
be the cone over $\Gamma$ with apex $v$, and
let $\calf$ be a (minimal) generating set of~$I_{\Gamma,d}$.
Then 
$I_{{\rm cone}_{v}(\Gamma), d_{V \cup \{v\}}}$ is (minimally) generated by
\[
{\rm cone}_{v}(\calf)  =  \left\{ p_{i_{1},j} \cdots p_{i_{n},j} -    p_{i_{1}',j} \cdots p_{i_{n}',j}  :   p_{i_{1}} \cdots p_{i_{n}} -    p_{i_{1}'} \cdots p_{i_{n}'} \in \calf \mbox{ and } j \in [d_{v}] \right\}.
\]
\end{enumerate}   
\end{lemma}

\begin{proof}$\quad$ \newline \indent
  (1) According to the dimension formula (\ref{eq:dimformula}), $I_{S_{V}, {\bf
      2}}$ is generated by a single equation.  The proof of (\ref{eq:dimformula})
  in \cite{Hosten2002} shows that the given binomial generates the ideal.

  (2) This follows because one can rearrange the rows and columns of $A_{{\rm
      cone}_{v}(\Gamma), d_{V \cup \{v\}}}$ so that it is a block diagonal matrix
  with $d_{v}$ diagonal blocks with the matrix $A_{\Gamma, d}$ along the diagonal.
  This decomposition appears in~\cite{Hosten2007}.
\end{proof}

\begin{ex}[Binary four-cycle]
  Let $C$ be a four-cycle with edges $12, 13, 24, 34$.  The cycle decomposes as
  the union of two paths with edges $12, 13$ and $24, 34$.  With $V_{1} =
  \{1,2,3\}$ and $V_{2} = \{2,3,4\}$, $I_{C, {\bf 2}}$ is the toric
  fiber product of $I_{C_{V_{1}}, {\bf 2}}$ and $I_{C_{V_{2}}, {\bf 2}}$.
  According to Lemma~\ref{lem:cones}, the Markov basis of a
  path of length three with edges $12,13$, consists of the two elements
$$
\begin{bmatrix}
i & 1 & 1 \\
i & 2 & 2 \\
\end{bmatrix} -
\begin{bmatrix}
i & 1 & 2 \\
i & 2 & 1 \\
\end{bmatrix},  \quad  \quad  i \in \{1,2\}.
$$
Similarly, the Markov basis for the path with edges $24,34$ consists of the two elements
$$
\begin{bmatrix}
 1 & 1 & l \\
 2 & 2  & l \\
\end{bmatrix} -
\begin{bmatrix}
 1 & 2 & l \\
 2 & 1 & l\\
\end{bmatrix},  \quad  \quad  l \in \{1,2\}.
$$
These Markov bases are slow-varying with respect to the codimension one toric
fiber product obtained by the overlap complex, which is two isolated vertices $2,
3$.  The vector $\bfh$ for the complex of two isolated vertices is
$$
\begin{bmatrix}
 1 & 1 \\
 2 & 2  \\
\end{bmatrix} -
\begin{bmatrix}
 1 & 2  \\
 2 & 1 \\
\end{bmatrix}.
$$
 The glue operation on these Markov bases produces four moves:
$$
\begin{bmatrix}
i & 1 & 1 & l \\
i & 2 & 2 & l \\
\end{bmatrix} -
\begin{bmatrix}
i & 1 & 2 & l \\
i & 2 & 1 & l \\
\end{bmatrix},  \quad  \quad  i,l \in \{1,2\}.
$$
The associated codimension zero toric fiber product is the hierarchical model
associated to the complex $\Gamma = C \cup \{\{2,3\}\}$, two triangles glued along
an edge.  It produces four quadratic elements of ${\rm Quad}:$
$$
\begin{bmatrix}
1 & j & k & 1 \\
2 & j & k & 2 \\
\end{bmatrix} -
\begin{bmatrix}
1 & j & k & 2 \\
2 & j & k & 1 \\
\end{bmatrix},  \quad  \quad  j,k \in \{1,2\}.
$$
A triangle with edges $12,13,23$ has a single quartic move in its Markov basis,
which is:
$$
\begin{bmatrix}
1 & 1 & 1 \\
1 & 2 & 2 \\
2 & 1 & 2 \\
2 & 2 & 1
\end{bmatrix}
- 
\begin{bmatrix}
1 & 1 & 2 \\
1 & 2 & 1 \\
2 & 1 & 1 \\
2 & 2 & 2
\end{bmatrix}.
$$
Lifting this move produces $16$ quartic Markov basis elements:
$$
\begin{bmatrix}
1 & 1 & 1 & l_{1} \\
1 & 2 & 2 & l_{2}\\
2 & 1 & 2 & l_{3}\\
2 & 2 & 1 & l_{4}
\end{bmatrix}
- 
\begin{bmatrix}
1 & 1 & 2 & l_{3}\\
1 & 2 & 1 & l_{4}\\
2 & 1 & 1 & l_{1}\\
2 & 2 & 2 & l_{2}
\end{bmatrix}  \quad \quad  l_{1}, l_{2}, l_{3}, l_{4} \in \{1,2\}.
$$
Similarly, the lifting operation from the cycle with edges $23,24,34$ produces
$$
\begin{bmatrix}
i_{1} & 1 & 1 & 1 \\
i_{2} & 1 & 2 & 2 \\
i_{3} & 2 & 1 & 2 \\
i_{4} & 2 & 2 & 1
\end{bmatrix}
- 
\begin{bmatrix}
i_{1} & 1 & 1 & 2 \\
i_{2} & 1 & 2 & 1 \\
i_{3} & 2 & 1 & 1 \\
i_{4} & 2 & 2 & 2
\end{bmatrix}  \quad \quad  i_{1}, i_{2}, i_{3}, i_{4}  \in \{1,2\}.
$$
Theorem~\ref{thm:slowvarying} implies that the lifts of $8$ quadrics and $32$
quartics generate the~$I_{C, {\bf 2}}$.  However, these elements do not form a
minimal generating set.  Direct computation in {\tt 4ti2}~\cite{4ti2} shows that a
minimal Markov basis contains all 8 quadrics but only 8 of the quartics.
\end{ex}

Similar arguments and the description of Markov bases of small cycles in
Lemma~\ref{lem:smallgraphs} can be used to get an explicit description of Markov
bases of the four-cycles that appear in Example~\ref{ex:hillar}.  
We can also produce analogous results for higher dimensional complexes.

\begin{thm}\label{thm:bipyramid}
  Let $B_{n}$ be the simplicial complex with vertex set $[n+2]$ and minimal
  non-faces $[n]$ and~$\{n+1,n+2\}$. The ideal $I_{B_{n}, {\bf 2}}$ has a
  generating set consisting of binomials of degrees $2, 2^{n-1}, $ and $2^{n}$.
\end{thm}

\begin{proof}
  For $i=1,2$ let $\Gamma_i$ be the cone $B_n \setminus (n+i)$ over the boundary
  of the simplex on~$[n]$. Then $B_n = \Gamma_1 \cup \Gamma_2$.  According to part
  (1) of Lemma~\ref{lem:cones} the Markov basis of $I_{\Gamma_1 \cap \Gamma_2 ,
    {\bf 2}}$ consists of a single element of degree $2^{n-1}$; and according to
  part (2) the ideals $I_{\Gamma_{i}, {\bf 2}}$ are each generated by two
  binomials of degree $2^{n-1}$.  Since $2^{n-1} < 2 \times 2^{n-1}$, by
  Proposition~\ref{prop:slowvaryingcheck}, the Markov bases for $I_{\Gamma_{1},
    {\bf 2}}$ and $I_{\Gamma_{2}, {\bf 2}}$ are slow-varying with respect to the
  Markov basis of $I_{\Gamma_1 \cap \Gamma_2 , {\bf 2}}$.  The set of glue moves
  consists of $4$ binomials of degree $2^{n-1}$.

  The simplicial complex $\tilde{\Gamma}$ appearing in the associated codimension
  zero toric fiber product has $[n]$ as an additional face.  Consequently it
  consists of the boundaries of two $(n-1)$-dimensional simplices that share a
  single facet.  By part (1) of Lemma~\ref{lem:cones}, the Markov basis of the boundary
  of an $(n-1)$-dimensional simplex consists of a single element of degree
  $2^{n}$.  The lifting operation preserves degree and produces $2^{2^{n}}$
  elements per boundary simplex, for a total of $2^{2^{n} +1}$ elements of
  degree~$2^{n}$.  Finally, there are $2^{n}$ quadrics in~${\rm Quad}$.
  Theorem~\ref{thm:slowvarying} shows that the union of all these elements is a
  Markov basis.
\end{proof}

The simplicial complex $B_{n}$ is the boundary of the polytope that is a bipyramid
over a simplex.  In particular, it is a simplicial sphere.
Theorem~\ref{thm:bipyramid} and the results of \cite{Petrovic2009} provide
evidence for the following conjecture.

\begin{conj}\label{conj:triMa}
  Let $\Gamma$ be a triangulation of a sphere of dimension~$n$.  Then the Markov
  basis of $I_{\Gamma, {\bf 2}}$ consists of elements of degree at most $2^{n+1}$.
\end{conj}  

To conclude this section, we give an example which shows how the gluing operation
can produce Markov basis elements of larger degree than either of the constituent
binomials.

\begin{ex}
  Let $G$ be the graph with vertex set $[5]$ and all edges except $1$--$5$, and assume
  that $d = {\bf 2}$.  Thus, $G$ consists of two $K_{4}$ graphs glued along an
  empty triangle.  The Markov basis for $K_{4}$ consists of $20$ elements of
  degree four and $40$ elements of degree six.  The overlap triangle is the
  boundary of a simplex, whose Markov basis consists of a single element of degree
  four.  Since $6 < 2 \times 4$, by Proposition~\ref{prop:slowvaryingcheck} the
  Markov bases of each of the $K_{4}$ are slow-varying.  Consider the following
  two binomials in the ideal of $K_{4}$:
  \begin{equation*}
    \begin{bmatrix}
      1 & 1 & 1 & 1 \\
      2 & 1 & 2 & 2 \\
      2 & 2 & 1 & 2 \\
      2 & 2 & 2 & 1 \\
      \hline
      1 & 1 & 1 & 1 \\
      1 & 2 & 2 & 2 \\
    \end{bmatrix}
    - 
    \begin{bmatrix}
      1 & 1 & 1 & 2 \\
      1 & 1 & 2 & 1 \\
      1 & 2 & 1 & 1 \\
      2 & 2 & 2 & 2 \\
      \hline
      2 & 1 & 1 & 1 \\
      2 & 2 & 2 & 2
    \end{bmatrix}, \quad \quad
    \begin{bmatrix}
      1 & 1 & 1 & 1 \\
      1 & 2 & 2 & 1 \\
      2 & 1 & 2 & 1 \\
      2 & 2 & 1 & 2 \\
      \hline
      1 & 1 & 2 & 2 \\
      2 & 2 & 1 & 2 \\
    \end{bmatrix}
    - 
    \begin{bmatrix}
      1 & 1 & 2 & 1 \\
      1 & 2 & 1 & 2 \\
      2 & 1 & 1 & 2 \\
      2 & 2 & 2 & 2 \\
      \hline
      1 & 1 & 2 & 1 \\
      2 & 2 & 1 & 1
    \end{bmatrix}.
  \end{equation*}
  The first sextic comes from the $K_{4}$ on vertex set $\{1,2,3,4\}$ and the
  second one from the $K_{4}$ on vertex set $\{2,3,4,5\}$.  In the columns
  corresponding to $\{2,3,4\}$ they agree in the first four rows and disagree in
  the last two rows.  This means that upon gluing these sextics, we produce moves
  of degree $4 + 2 + 2 = 8$.  In particular we get
  \begin{equation*}
    \begin{bmatrix}
      1 & 1 & 1 & 1 & 1\\
      2 & 1 & 2 & 2 & 1\\
      2 & 2 & 1 & 2 & 1\\
      2 & 2 & 2 & 1 & 2\\
      \hline
      1 & 1 & 1 & 1 & m_{1}\\
      1 & 2 & 2 & 2 & m_{2}\\
      \hline
      i_{1} & 1 & 1 & 2 & 2 \\
      i_{2} & 2 & 2& 1 & 2
    \end{bmatrix}
    - 
    \begin{bmatrix}
      1 & 1 & 1 & 2 & 1 \\
      1 & 1 & 2 & 1 & 2 \\
      1 & 2 & 1 & 1 & 2 \\
      2 & 2 & 2 & 2 & 2 \\
      \hline
      2 & 1 & 1 & 1 & m_{1} \\
      2 & 2 & 2 & 2 & m_{2} \\
      \hline
      i_{1} & 1 & 1 & 2 & 1 \\
      i_{2} & 2 & 2 & 1 & 1
    \end{bmatrix},  \quad \quad  i_{1}, i_{2}, m_{1}, m_{2} \in \{1,2\}.
  \end{equation*}
  In this example gluing yields degrees four, six, and eight.  Lifting produces
  Markov basis elements of degrees four and six.  Direct computation with {\tt
    4ti2} shows, however, that a minimal Markov basis of this model contains only
  binomials of degree two, four, and six.  Therefore the gluing operation may
  produce elements of unnecessarily large degree.
\end{ex}


\subsection{Cycles and ring graphs}
In this subsection, and the next, $\Gamma = G$ is a graph.  We start with cycles
and graphs that can be easily constructed from cycles, then explore $K_{4}$-minor
free graphs, providing a new proof of the main result in~\cite{Kral2008}.  To set
up induction we provide the Markov bases of simple graphs.

\begin{lemma}\label{lemma:pathFree}
  Let $P$ be a path and $d\in\zz^{V}_{\geq 2}$ arbitrary, then $\mu(P,d) = 2$.
\end{lemma}
\begin{proof}
  This follows from Theorem~\ref{thm:codimzero} or the results on decomposable
  simplicial complexes in~\cite{Dobra, TakkenThesis}.
\end{proof}

\begin{lemma}[Small Graphs]\label{lem:smallgraphs}$ $\\[-3ex]
\begin{enumerate}
\item Let $K_{3}$ be the triangle.  The following table contains known values of
  $\mu(K_{3},d)$:
\[ 
\begin{array}{c|cccccc}
d_1 & 2 & 3 & 3 & 3        & 3 & 4        \\
d_2 & p & 3 & 3 & 3        & 4 & 4      \\
d_3 & q & 3 & 4 & q \geq 5 & 4 & 4 \\
\hline 
\mu(I_{K_{3},d}) & \min(2p,2q) & 6 & 8 & 10 & 12 & 14 \ \\ 
\end{array} 
  \]
\item If $C$ is a four-cycle with edges $12,23,34,41$, then $\mu(C,d)$ takes the
  following values
\[ \begin{array}{c|cccccccc}
d_1 & 2 & 2 & 2 & 2 & 2 & 2 & 2 & 3 \\
d_2 & 2 & 2 & 2 & 2 & 2 & 2 & 3 & 3\\
d_3 & 3 & 3 & 3 & 4 & 4 & 5 & 3 & 3\\
d_4 & 3 & 4 & 5 & 4 & 5 & 5 & 3 & 3\\
\hline 
\mu(C,d)  & 6 & 6 & 6 & 8 & 8 & 10 & 6 & 6 \\ 
\end{array} \]  
\item If $C$ is a five-cycle with edges $12,23,34,45,51$, $d_1=d_2=2$, and
  $d_3=d_4=d_5=3$, then $\mu(C,d) = 6$.
\item Let $K_{2,3}$ be the complete bipartite graph on $\{1,2\}$ and
  $\{3,4,5\}$. If $d_{1} = d_{2} = 3$ and $d_{3} = d_{4} = d_{5} = 2$, then
  $\mu(K_{2,3},d) = 6$.
\item The complete graph $K_{4}$ with $d = {\bf 2}$ satisfies $\mu(K_{4},{\mathbf 2}) = 6$.
\end{enumerate}
\end{lemma}

\begin{proof}
  The computation for $K_{3}$ with $d = (2,p,q)$ is contained in the original work
  of Diaconis and Sturmfels~\cite{Diaconis1998}.  The values for $d = (3,3,q)$
  were determined by Aoki and Takemura~\cite{Aoki2003}.  All other values have
  been computed using {\tt 4ti2} and Markov bases are available on the Markov
  Basis Database~\cite{Kahle2008}.
\end{proof}

\begin{lemma}\label{lemma:s6l1}
  Let $G$ be a graph, and $V_{1}, V_{2} \subseteq V$ such that $V_{1} \cup V_{2} =
  V$, $G = G_{V_{1}} \cup G_{V_{2}}$ and either $V_{1} \cap V_{2} = \{u\}$ or
  $V_{1} \cap V_{2} = \{u, v\}$ with $uv$ an edge of~$G$.  Then
  \[  
  \mu( G, d ) = \max ( 2, \mu(G_{V_{1}},d_{V_{1}}), \mu(G_{V_{2}}, d_{V_{2}}) ).
  \] 
\end{lemma}
\begin{proof}
  In either case $I_{G,d}$ is a codimension zero toric fiber product and
  Theorem~\ref{thm:codimzero} applies.  The statement also follows from results on
  reducible hierarchical model in~\cite{Dobra2004, Hosten2002, Sullivant2007}.
\end{proof}

\begin{lemma}\label{lemma:s6l2}
Let $G$ be a graph, and $V_{1}, V_{2} \subseteq V$ such that $V_{1} \cup V_{2} = V$, $G = G_{V_{1}} \cup G_{V_{2}}$ and $V_{1} \cap V_{2}  = \{u, v\}$ where $uv$ is not an edge of $G$, and suppose that $d_{u} = d_{v} = 2$.  Further suppose that $\mu(G_{V_{1}}, d_{V_{1}}) = \mu(G_{V_{2}}, d_{V_{2}}) = 2$.   Then
\[ \mu(G, d) \leq \max ( 2, \mu(G_{V_{1}} \cup\, uv, d_{V_{1}}), \mu(G_{V_{2}} \cup\, uv, d_{V_{2}}) ). \]
\end{lemma}

\begin{proof}
  The intersection of $G_{V_{1}}$ and $G_{V_{2}}$ is the graph with two nodes, and
  no edges.  Since $d_{u} = d_{v} = 2$, the dimension formula
  (\ref{eq:dimformula}) implies that this is a codimension one toric fiber
  product.  The toric ideal of the graph consisting of two 
  isolated nodes, and $d_{u} = d_{v} = 2$ is generated by a single
  quadratic binomial, by Lemma~\ref{lem:cones}.  Furthermore, the fact that $\mu(G_{V_{1}}, d_{V_{1}}) =
  \mu(G_{V_{2}}, d_{V_{2}}) = 2$, and that hierarchical models have no Markov
  basis elements of degree one, implies that the Markov bases of
  $I_{G_{V_{1}},d_{V_{1}}}$ and $I_{G_{V_{2}},d_{V_{2}}}$ are slow-varying,
  by Proposition~\ref{prop:slowvaryingcheck}.
  Hence Theorem~\ref{thm:slowvarying} shows that the Markov basis of
  $I_{G,d}$ consists of the glued elements of the Markov bases of
  $I_{G_{V_{1}},d_{V_{1}}}$ and $I_{G_{V_{2}},d_{V_{2}}}$, together with the
  Markov basis of the associated codimension zero toric fiber product, which is
  \[
  I_{G \cup\, uv,d} = I_{G_{V_{1}} \cup\, uv,d_{V_{1}}} \tfp I_{G_{V_{2}} \cup\,
    uv, d_{V_{2}}},
  \] 
  by Proposition~\ref{prop:gammaassoc0}.  Since we only ever glue quadrics along a
  quadric, the resulting binomial is also of degree~two.  The generators of the
  associated codimension zero toric fiber product consists of quadratic elements
  and lifts of generators of $I_{G_{V_{1}} \cup\, uv,d_{V_{1}}}$ and $I_{G_{V_{2}}
    \cup\, uv, d_{V_{2}}}$.  Since lifting preserves degrees, the quantity $\max ( 2,
  \mu(G_{V_{1}} \cup\, uv, d_{V_{1}}), \mu(G_{V_{2}} \cup\, uv, d_{V_{2}}))$ is
  the maximum degree of a generator of the associated codimension zero toric fiber
  product.
\end{proof}

\begin{lemma}\label{lemma:manyCycles}
  Let $C$ be a cycle with vertex set $V$ and $d\in\zz^{V}_{\geq 2}$.
\begin{itemize}
\item[(1)] If $C$ contains no edge $uv$ with $d_{u},d_{v} > 2$ then $\mu(C, d) = 4$.
\item[(2)] If all $d_v\leq 3$ and $C$ contains no path $u_1u_2u_3u_4$ with all
  $d_{u_i}>2$, then $\mu(C,d) \leq 6$.
\item[(3)] If all $d_v\leq 4$ and $C$ contains no path $u_1u_2u_3$ with all
  $d_{u_i}>2$, then $\mu(C,d) \leq 8$.
\item[(4)] If all $d_v\leq 5$ and $C$ contains no path $u_1u_2u_3$ with all
  $d_{u_i}>2$, then $\mu(C,d) \leq 10$.
\end{itemize}
\end{lemma}
\begin{proof}
  We give a detailed proof of (1).  According to Lemma~\ref{lem:smallgraphs} the
  statement holds for cycles of length three. We proceed by induction on the
  length of~$C$. There are always two non-adjacent vertices $u$ and $v$ in $C$
  with $d_u=d_v=2$. Let $V_{1}$ be the set of vertices on one of the paths in $C$
  from $u$ to $v$, and let $V_{2}$ be the set of vertices on the other path.
  According to Lemma~\ref{lemma:pathFree} the Markov width of paths is two. By
  induction we find $\mu(G_{V_{1}}\cup\, uv, d_{V_{1}}) = 4$ and $\mu(G_{V_{2}}
  \cup\, uv, d_{V_{2}}) = 4$, since those graphs are shorter cycles than $C$
  satisfying the conditions in~(1). By Lemma~\ref{lemma:s6l2}, the Markov width of
  $\mu(C,d) = 4$.  Statements (2)-(4) follow by the same inductive argument and
  reducing to the small graphs in Lemma~\ref{lem:smallgraphs}.
\end{proof}

Cycles can be patched together to form larger graph classes, for example ring graphs.

\begin{defn}
A \emph{ring graph} is a graph that can be recursively constructed from paths and cycles by disjoint unions, identifying a vertex of disjoint components, and identifying edges on disjoint components.  An \emph{outerplanar graph} is a graph with a planar embedding such that all vertices are on a circle.
\end{defn}

Outerplanar graphs are also characterized as the largest minor closed class that
excludes $K_{4}$ and~$K_{2,3}$.  This in particular implies that all outerplanar
graphs are series-parallel since they have no $K_4$-minors.  It is easy to see
that outerplanar graphs are ring graphs. Recall that a graph is
\emph{$k$-connected} if there is no way to disconnect it by removing at most $k-1$
vertices. We need to describe how to decompose 2-connected ring graphs into
cycles.

\begin{defn}
  A \emph{cycle decomposition} of a 2-connected ring graph $G$ is a sequence $C_1,
  C_2, \ldots, C_k$ of cycles in $G$ such that
\begin{itemize}
\item the union of all $C_i$ is $G,$ and
\item the intersection of $C_1\cup \cdots \cup C_{i}$ and $C_{i+1}$ is an edge for $1\leq i < k$.
\end{itemize}
\end{defn}

Any $2$-connected ring graph must have a cycle decomposition, since a
$2$-connected ring graph is obtained by only identifying edges in disjoint
components.

\begin{thm}\label{thm:s6main}
Let $G$ be a ring graph whose maximal 2-connected subgraphs are $G_1, G_2, \ldots, G_l$ and assume that $C_1^i, C_2^i, \ldots, C_{k_i}^i$ is a cycle decomposition of $G_i$ for all $1\leq i \leq l$. If for all $C=C^i_j,$
\begin{itemize}
\item[(1)] there is no edge $uv$ in $C$ with $d_{u},d_{v} > 2$ then $\mu(G,d) \leq
  4$. 
\item[(2)] all $d_v\leq 3$ and there is no path $u_1u_2u_3u_4$ in $C$ with all
  $d_{u_i}>2$, then $\mu(G,d) \leq 6$. 
\item[(3)] all $d_v\leq 4$ and there is no path $u_1u_2u_3$ in $C$ with all
  $d_{u_i}>2$, then $\mu(G,d) \leq 8$. 
\item[(4)] all $d_v\leq 5$ and there is no path $u_1u_2u_3$ in $C$ with all
  $d_{u_i}>2$, then $\mu(G,d) \leq
  10$. 
\end{itemize}
\end{thm}
\begin{proof}
  This follows directly from Lemma~\ref{lemma:pathFree}, Lemma~\ref{lemma:s6l1},
  and Lemma~\ref{lemma:manyCycles}.
\end{proof}

\begin{defn}
  A graph $G$ is \emph{Markov slim}, if for every independent set $I$ of $G$ the
  model with $d_v \geq 2$ for $v\in I$ and $d_v=2$ for $v\in V(G)\setminus I$ has
  Markov width at most four.
\end{defn}

\begin{thm}\label{thm:outerplanar}
  The maximal minor-closed class of Markov slim graphs is the outerplanar graphs.
\end{thm} 
\begin{proof}
  By Theorem~\ref{thm:s6main} the outerplanar graphs are Markov slim since they
  are ring graphs.  Say that there is a minor closed class larger than the
  outerplanar graphs, in which every graph is Markov slim.  Then this class
  either contains $K_4$ or~$K_{2,3}$. By parts (4) and (5) of
  Lemma~\ref{lem:smallgraphs} neither $K_{4}$ nor $K_{2,3}$ are Markov slim.
\end{proof}

Repeated toric fiber products of cycles reduce computations of the Markov width to
the three cycle.  Therefore the following conjecture seems natural.

\begin{conj}
Let $C$ be a cycle of length $n$, with edges $12, 23, \ldots, n1$.  Then the Markov width $\mu(C,d)$ equals
$$
\max_{i = 1, \ldots, n}  \mu(K_{3}, (d_{i}, d_{i+1}, d_{i+2}))
$$
where the indices $i, i+1, i+2$ are considered cyclically modulo~$n$.
\end{conj}

Our results so far only work with codimension one toric fiber products, which do
not raise the degree of generators in the cycle case, and hence we always glued
paths at a pair of vertices $u,v$ where $d_{u} = d_{v} = 2$.  It is not clear
whether or not this remains true for larger values of $d_{u}, d_{v}$.


\subsection{Binary series-parallel graphs}\label{sec:seriespar}

To prove Theorem~\ref{thm:kral} we apply a classical decomposition of
$K_{4}$-minor free graphs.
\begin{defn}
  The class $SP$ of \emph{connected series-parallel graph} is the smallest
  collection of graphs satisfying the following properties.
\begin{itemize}
\item Each graph $G \in SP$ has two distinguished vertices, the top and the bottom
  vertex, which are different.
\item The graph $K_{2}$ is in~$SP$.
\item If $G_{1}$ and $G_{2}$ are in $SP$ with tops and bottoms $t_{1}, t_{2}$,
  $b_{1}, b_{2}$ respectively, then
\begin{description}
\item[Series construction] the graph obtained from $G_{1}$ and $G_{2}$ by
  identifying $t_{1}$ and $b_{2}$ and calling $b_{1}$ and $t_{2}$ the new bottom
  and top also belongs to $SP$;
\item[Parallel construction] the graph obtained form $G_{1}$ and $G_{2}$ by
  identifying $t_{1}$ and $t_{2}$ and $b_{1}$ and $b_{2}$ (and calling these the
  new top and bottom) is also in~$SP$.
\end{description}  
\end{itemize}
\end{defn}

In a graph without $K_{4}$-minors, every 2-connected component is a
series-parallel graph (see \cite[Chapter~7]{Diestel2006}).  Since gluing two
graphs at a vertex is a codimension zero toric fiber product, to prove
Theorem~\ref{thm:kral}, we can restrict to series-parallel graphs.  One tool is
the following lemma about choices that can be made in the parallel construction.

\begin{lemma}\label{lem:parallel}
  Suppose that $G \in SP$ has at least four vertices.  Then $G$ can be obtained by
  series or parallel construction from two graphs $G_{1}$ and $G_{2}$ each with
  fewer vertices than~$G$.
\end{lemma}

\begin{proof}
  The series construction of $G_{1}$ and $G_{2}$ clearly produces a graph $G$ with
  a larger number of vertices.  For the parallel construction, if both $G_{1}$ and
  $G_{2}$ are not single edges then their parallel construction has more vertices
  than either $G_{1}$ or $G_{2}$.  The only non-trivial case is when one of the
  two graphs, say $G_{1}$, is a single edge.

  We can assume $G_{2}$ is neither a path of one or two edges, nor $K_{3}$, since
  then the resulting graph would have less than three vertices.  The graph $G_{2}$
  is obtained either by a series or by a parallel construction from two graphs
  $G_{3}$ and~$G_{4}$.  In the case of a parallel construction, consider new
  graphs $\tilde{G_{3}}$ and $\tilde{G_{4}}$ with an edge glued in from $t$ to $b$
  in both cases.  The resulting parallel construction of $\tilde{G_{3}}$ and
  $\tilde{G_{4}}$ gives the same graph as the parallel construction of $G_{1}$ and
  $G_{2}$.  In the case of a series construction, one of the graphs $G_{3}$ or
  $G_{4}$ has $\geq 3$ vertices.  Assume that graph is $G_{4}$.  A series
  construction of $G_{1}$ with $G_{3}$ followed by a parallel construction of the
  result with $G_{4}$ gives the original graph.  We may have to rearrange the tops
  and bottoms during this construction, but doing so does not change the property
  of being a series-parallel graph.
\end{proof}

\begin{thm}\label{thm:series-parallel}
  If $G$ is a connected series-parallel graph with top $t$ and bottom~$b$, then
  $\mu(G, {\bf 2}) = 4$ and a Markov basis of $I_{G, {\bf 2}}$ can be chosen to
  consist of:
\begin{itemize}
\item[(1)] Degree four binomials whose terms have the same degree
on the $bt$ subcomplex.
\item[(2)] Degree two binomials that are slow-varying on the $bt$ subcomplex.
\end{itemize}
\end{thm}

\begin{proof}
  We proceed by induction on the number of vertices of the graph.  The statement
  is trivially true for connected series-parallel graphs with one or two vertices,
  since they have empty Markov basis.  There are two graphs with three vertices to
  consider.  For the triangle $I_{K_{3}, {\bf 2}}$ there is one degree four
  generator and it must project to the zero polynomial along the $bt$ edge, since
  that edge belongs to~$K_{3}$.  In the case of the path with three vertices,
  there are two quadratic generators, which are slow-varying by
  Proposition~\ref{prop:slowvaryingcheck}.

  Now let $G$ be a series-parallel graph with at least four vertices.  By
  Lemma~\ref{lem:parallel} it can be built from two graphs $G_{1}$ and $G_{2}$
  with strictly smaller numbers of vertices by either a series or a
  parallel construction.  We must show that properties $(1)$ and $(2)$ of the
  Markov basis are preserved under either of these constructions.

  First suppose that $G$ is obtained from $G_{1}$ and $G_{2}$ by a series
  construction.  There are three types of generators that arise.  The generators
  are given by:
\begin{description}
\item[Lift 1] lifting generators from $I_{G_1, {\bf 2}}$ while being constant on $G_2$;
\item[Lift 2] lifting generators from $I_{G_2, {\bf 2}}$ while being constant on $G_1$;
\item[Quad] quadratic moves.
\end{description}
Since Lifting preserves degrees we obtain only moves of degree two and four.
Quadratic moves are slow-varying by Proposition~\ref{prop:slowvaryingcheck}, thus
we must show that the degree four moves can be chosen so that their projections on
the $bt$ edge are constant.  The crucial idea is that the degree four generators
all come from three-cycles, since 
we are always only using series or parallel construction.  The quartic generator
for $I_{K_{3}, {\bf 2}}$ is
\[
\begin{bmatrix}
1 & 1 & 1 \\
1 & 2 & 2 \\
2 & 1 & 2 \\
2 & 2 & 1
\end{bmatrix}
- 
\begin{bmatrix}
1 & 1 & 2 \\
1 & 2 & 1 \\
2 & 1 & 1 \\
2 & 2 & 2
\end{bmatrix}.
\]
Any subsequent appearance of a quartic is a lift of this move in some way and must
be obtained by using a single edge or vertex in $K_{3}$ and performing a sequence
of lifts.  The pair $bt$ cannot go from an added vertex to the third vertex of the
underlying $K_{3}$, otherwise we would be able to construct graphs that have
$K_{4}$ as a minor.  Thus, $b$ and $t$ belong to the gluing edge, or a subset of
the lifted vertices.  However, by construction of the lift operation, the binomial
projects to zero when restricted to such a subset of vertices.

If $G$ is obtained from a parallel construction of $G_1$ and $G_2$, then the top
and bottom vertices can be adjacent or not. If they are adjacent, then we are
gluing along an edge.  All generators of $I_{G_{1}, {\bf 2}}$ and $I_{G_{2}, {\bf
    2}}$ project to zero along this edge by properties of the lift operation.  If
the special vertices are not adjacent, we have a codimension one toric fiber
product.  The associated codimension zero product consists of series-parallel
graphs with fewer vertices.  By the argument in the preceding paragraphs, all
Markov basis elements obtained from the associated codimension zero toric fiber
product satisfies either (1) or~(2).  Finally, consider ${\GGlue}(\calf, \calg)$.
Since all Markov bases satisfy (1) and (2), we only ever glue quadrics, producing
more quadrics, which are slow-varying by Proposition~\ref{prop:slowvaryingcheck}.
\end{proof}

Instead of using binary variables for the triangle in the proof, one could have
used larger values of $d_{v}$ on the vertex of the triangle that is never involved
in gluing or identification. This would have given an alternative but less
descriptive proof of Theorem~\ref{thm:s6main}.  The procedure yields a larger
class than ring graphs, but it is not true that larger $d_{v}$ on independent sets
always produce Markov width four, as illustrated earlier by the fact that
$K_{2,3}$ is not Markov slim.

There are further applications of higher codimension toric fiber products in
algebraic statistics lurking.  For example, ideals of graph
homomorphisms~\cite{Engstrom2010} generalize classes of toric ideals in algebraic
statistics. Given graphs $G$ and $H$, potentially with loops, the ideal of graph
homomorphisms from $G$ to $H$ is $I_{G \rightarrow H}$.  In this language, binary
hierarchical models arise as the special case where $H = K_{2}^{o}$ is the
complete graph with loops.  If $H$ is an edge with one loop, then the
homomorphisms from $G$ to $H$ correspond to the independent sets of~$G$. It is
known that $I_{G \rightarrow H}$ is quadratically generated if $G$ is bipartite,
or becomes bipartite after the removal of one vertex~\cite{Engstrom2010}. Using
Theorem~\ref{thm:series-parallel} as a template, one derives that $I_{G
  \rightarrow H}$ is quadratically generated for series-parallel~$G$.
 
Some toric ideals are not toric fiber products themselves, but project to one.
With control over the projection one may be able to find a Markov basis anyway.
An Example is Nor\'en's proof of a conjecture by Haws, Martin del Campo, Takemura,
and Yoshida~\cite{Noren2012}.


\section{Application: conditional independence ideals}\label{sec:ci}
A basic problem in the algebraic study of conditional independence is to
understand primary decompositions of CI-ideals.  For instance, if a conditional
independence model comes from a graph, the minimal primes provide information
about families of probability distributions that satisfy the
conditional independence constraints but do not factorize according to the graph.
Moreover, primary decompositions can provide information about the connectivity of
random walks using Markov subbases~\cite{KRS12}.

In this section $J$ is the generic letter denoting an ideal.  This is to avoid
confusion between the ideals $I_{G}$ of Section~\ref{sec:markov} and the CI-ideals
$J_{G}$ in Section~\ref{sec:graphical-ci-models}.  The results in this section are
independent of $d = (d_{v})_{v\in V}$, the vector of cardinalities.  It is fixed
arbitrarily and does not appear in the notation.

Assume $\mm$ is a conditional independence model and $J_{\mm}$ its CI-ideal.  Our
goal is to describe conditions under which there exist suitable conditional
independence models $\mm_{1}$ and $\mm_{2}$ such that
\begin{equation*}
  J_{\mm} = J_{\mm_{1}} \tfp J_{\mm_{2}}.
\end{equation*}
When $J_{\mm}$ is as a toric fiber product, the results of
Section~\ref{sec:pers-prim-decomp} yield a primary decomposition of $J_{\mm}$ from
primary decompositions of $J_{\mm_{1}}$ and $J_{\mm_{2}}$, greatly reducing the
necessary computational efforts.  This seems to work best in the case of
codimension zero toric fiber products.  At this moment it is not clear if there is
a use for higher codimension toric fiber products in analyzing conditional
independence models.

We first develop a general theory for arbitrary conditional independence models.
Then we apply it to global Markov ideals of graphs, showing that they are toric
fiber products if the graph has a decomposition along a clique.

We assume the same setup as in Section~\ref{sec:2:hierarchical} for hierarchical
models.  Let $A,B,C$ be three pairwise disjoint subsets of $V$, and $D\defas
V\setminus (A\cup B \cup C)$.  If $D\neq \emptyset$, then $p_{i_{A}i_{B}i_{C}+}
\defas \sum_{i_{D} \in \D_{D}} p_{i_{A}i_{B}i_{C}i_{D}}$.  The~\emph{conditional
  independence (CI) ideal} $J_{\ind{A}{B}[C]}$ is
\begin{equation*}
  J_{\ind{A}{B}[C]} \defas \< p_{i_{A}i_{B}i_{C}+}p_{j_{A}j_{B}i_{C}+} - p_{i_{A}j_{B}i_{C}+}p_{j_{A}i_{B}i_{C}+}
  : i_{A},j_{A} \in \D_{A}, i_{B},j_{B}\in \D_{B}, i_{C}\in \D_{C}\>.
\end{equation*}
An argument similar to that in Section~\ref{sec:introCI} shows that this ideal is
prime. For a collection
\[
\mm = \{\ind{A_{1}}{B_{1}}[C_{1}], \ldots, \ind{A_{r}}{B_{r}}[C_{r}] \}
\]
of CI-statements, the CI-ideal is the sum of the ideals of its statements:
\[
J_{\mm}  =  J_{\ind{A_{1}}{B_{1}}[C_{1}]} + \cdots + J_{\ind{A_{r}}{B_{r}}[C_{r}]}.
\]

In statistics one is usually not interested in all of the variety of a CI-ideal,
but only its intersection with the set of probability distributions. The following
properties of CI-ideals imply well-known properties of conditional
independence.

\begin{prop}\label{prop:CI-inference}
  The following ideal containments hold:
  \begin{itemize}
  \item $J_{\ind{A}{B}[C]} = J_{\ind{B}{A}[C]}$ (symmetry);
  \item $J_{\ind{A}{B\cup D}[C]} \supset J_{\ind{A}{B}[C]}$ (decomposition);
  \item $J_{\ind{A}{B\cup D}[C]} \supset J_{\ind{A}{B}[C\cup D]}$ (weak union).
  \end{itemize}
\end{prop}
However, the contraction property does not hold algebraically since
\begin{equation*}
 J_{\ind{A}{B}[C\cup D]} + J_{\ind{A}{D}[C]} \not\supseteq J_{\ind{A}{B\cup D}[C]}.
\end{equation*}
The algebraic structure of $J_{\ind{A}{B}[C\cup D]} +
J_{\ind{A}{D}[C]}$ 
was analyzed systematically in~\cite{Garcia2005}.

\subsection{Toric fiber products of CI-models}
Let $\mm_{1}, \mm_{2}$ be conditional independence models on two (not necessarily
disjoint) sets of variables $V_{1}, V_{2}$, respectively. The CI-ideals
$J_{\mm_{1}}$ and $J_{\mm_{2}}$ live in polynomial rings with variables indexed by
$\D_{V_{1}}$, and $\D_{V_{2}}$, respectively.  Their toric fiber product is again
a CI-ideal when certain conditions are satisfied.  Our aim is to define the toric
fiber product of $J_{\mm_{1}}$ and $J_{\mm_{2}}$ combinatorially, using
CI-statements.

\begin{defn}[The $S$-grading]\label{defn:separator}
  Let $S\subset V$.  The grading on the polynomial ring $\kk[p_{i} : i \in
  \D_{V}]$ given by $\deg(p_{i}) = e_{i_{S}} \in \zz^{\D_{S}}$ is the
  \emph{$S$-grading}. The conditional independence model $\mm$ is
  \emph{$S$-homogeneous} if each statement $\ind{A}{B}[C]$ in $\mathcal{M}$
  satisfies either $S\subseteq (A\cup C)$ or $S \subseteq (B \cup C)$.
\end{defn}

\begin{lemma}\label{lem:sep-combinatorics}
  If $\mm$ is $S$-homogeneous then $J_{\mm}$ is homogeneous in the $S$-grading.
\end{lemma}

\begin{proof}
Let $D=V\setminus (A\cup B\cup C)$. The polynomial
\[f = p_{i_{A}i_{B}i_{C}+}p_{j_{A}j_{B}i_{C}+} -
p_{i_{A}j_{B}i_{C}+}p_{j_{A}i_{B}i_{C}+}\] is not homogeneous if $S \cap D \neq
\emptyset$ since expressions like $p_{i_{A}i_{B}i_{C}+}$ involve sums over terms
with different $D$-degrees.  Assuming that $S \cap D = \emptyset$, the degree of
all terms in the polynomial $p_{i_{A}i_{B}i_{C}+}p_{j_{A}j_{B}i_{C}+}$ is $e_{i_{A
    \cap S} i_{B\cap S} i_{C \cap S}} + e_{j_{A \cap S} j_{B\cap S} i_{C \cap
    S}}$.  The degree of all terms in $p_{i_{A}j_{B}i_{C}+}p_{i_{A}j_{B}i_{C}+}$
is $e_{i_{A \cap S} j_{B\cap S} i_{C \cap S}} + e_{j_{A \cap S} i_{B\cap S} i_{C
    \cap S}}$.  These two degrees are equal if and only if $S \subseteq A \cup C$
or $S \subseteq B \cup C$.
\end{proof}

\begin{ex}[Homogeneity with respect to the $S$-grading]
  Consider binary random variables $V = \set{1,2,3}$, where $S=\set{1}$. The
  statement $\ind{2}{3}$ is given by the polynomial
  \[
  (p_{111} + p_{211})(p_{122}+p_{222}) - (p_{112} + p_{212})(p_{121}+p_{221})
  \]
  which is not homogeneous in the $S$-grading. In contrast, the polynomial
  for~$\ind{1}{2}$,
  \[
  (p_{111} + p_{112})(p_{221}+p_{222}) - (p_{121} + p_{122})(p_{211}+p_{212}),
  \]
  is homogeneous of multidegree $e_{1} + e_{2}$.  
\end{ex}

The following example shows how redundant statements can seemingly complicate the
situation and why it is advantageous to work with minimal sets of CI-statements
defining a given CI-ideal.  However, solving the conditional independence
implication problem is difficult in general~\cite{Geiger1993}.
\begin{ex}
  The converse of Lemma~\ref{lem:sep-combinatorics} need not hold.
  Consider the ideal $J_{\ind{12}{3}}$, which is $\{1\}$-homogeneous.  By
  Proposition~\ref{prop:CI-inference} it equals the CI-ideal of $\mm =
  \{\ind{12}{3}, \ind{2}{3}\}$ which does not satisfy the combinatorial conditions
  in Lemma~\ref{lem:sep-combinatorics}.  
\end{ex}

Our next goal is to define the toric fiber product of two $S$-homogeneous
conditional independence models $\mm_{1}, \mm_{2}$ where $S = V_{1}\cap V_{2}$.
To this end, consider the statement
\[
  \mathcal{S} = \set {\ind{(V_{1}\setminus S)}{(V_{2}\setminus S)}[S]},
\]
representing a separating property of~$S$.  A second class of statements appearing
in the toric fiber product of $\mm_{1}$ and $\mm_{2}$ comes from joining vertices
in $V_{2}$ to statements in $\mm_{1}$ and vice versa.  By $S$-homogeneity and
symmetry in Proposition~\ref{prop:CI-inference} we can assume that each statement
$\ind{A}{B}[C]$ in $\mm_{1},\mm_{2}$ satisfies $A \cap S = \emptyset$ and define
\begin{equation}
  \label{eq:liftci}
  \begin{gathered}
    \mathcal{L}_{1} = \set{\ind{A}{B\cup(V_{2}\setminus S)}[C] :
      \ind{A}{B}[C]\in\mm_{1}},  \\
    \mathcal{L}_{2} = \set{\ind{A}{B\cup(V_{1}\setminus S)}[C] :
      \ind{A}{B}[C]\in\mm_{2}}.
  \end{gathered}
\end{equation}
The CI-statements in~\eqref{eq:liftci} are constructed so that their ideal
generators are exactly the lifts of ideal generators associated to the statements
in $\mm_{1}$ and $\mm_{2}$.  The straightforward definition of $\Lift$ in the
non-binomial case is contained in~\cite{Sullivant2007}.

\begin{lemma}
  \label{lem:liftinglemma} $ J_{\mathcal{L}_{i}} = \langle \Lift (\mathcal{M}_{i}) \rangle.$
\end{lemma}

\begin{proof}
  We only show the argument for $\mm_{1}$. Denote $D = V_{2}\setminus S$.  Lifting
  a polynomial
  \begin{equation*}
    p_{i_{A}i_{B}i_{C}+}p_{j_{A}j_{B}i_{C}+} -
    p_{i_{A}j_{B}i_{C}+}p_{j_{A}i_{B}i_{C}+} \in J_{\mm_{1}},
  \end{equation*}
  consists of choosing two configurations $i_{D}, j_{D} \in \D_{D}$, and lifting
  to:
  \begin{equation}
    \label{eq:liftedmarginal}
    q_{i_{A}i_{B}i_{C}i_{D}+}q_{j_{A}j_{B}i_{C}j_{D}+} -
    q_{i_{A}j_{B}i_{C}j_{D}+}q_{j_{A}i_{B}i_{C}i_{D}+} \in \Lift (\mm_{1}),
  \end{equation}
  where $i_{D}, j_{D}$ align with the configurations $i_{B}$ and $j_{B}$ by our
  convention that $S\subset B\cup C$. The lift~\eqref{eq:liftedmarginal}
  originates from one of the statements in $\mathcal{L}_{1}$ and every statement
  there produces generators of the given form.
\end{proof}

\begin{defn}
 The CI-model on $V_{1} \cup V_{2}$ given by all derived statements
  \begin{equation*}
    \begin{gathered}
      \mathcal{M}_{1} \tfps \mathcal{M}_{2} \defas \mathcal{S} \cup
      \mathcal{L}_{1} \cup \mathcal{L}_{2}.
    \end{gathered}
  \end{equation*}
  is the \emph{toric fiber product} of $\mm_{1}$ and $\mm_{2}$ along~$S$.
\end{defn}

\begin{thm} \label{sec:citfp-thm} For $i=1,2$ let $\mm_{i}$ be an $S$-homogeneous
  CI-model where $S=V_{1}\cap V_{2}$. If $\mathcal{A}$ is the linearly independent
  vector configuration representing the $S$-grading, then
  \begin{equation*}
    J_{\mm_{1} \tfps \mm_{2}} = J_{\mm_{1}} \tfp J_{\mm_{2}}.
  \end{equation*}
\end{thm}

\begin{proof} Homogeneity in the (codimension zero) $S$-grading follows from
  Lemma~\ref{lem:sep-combinatorics}.  The generators of the codimension zero toric
  fiber product on the right hand side consist of Lifts and Quads
  by~\cite{Sullivant2007} and, in the toric case, Section~\ref{sec:codim0}. Since
  the Quads correspond exactly to the independence statement $\mathcal{S}$, the
  theorem is a consequence of Lemma~\ref{lem:liftinglemma}.
\end{proof}

\begin{ex}
Let $V_{1} = \{1,2,3,4\}$ and $V_{2} = \{3,4,5,6\}$.  Let 
$$
\mm_{1} = \{\ind{1}{3}[\{2,4\}],  \ind{2}{4}[\{1,3\}] \} \mbox{ and } 
\mm_{2} = \{\ind{3}{5}[\{4,6\}],  \ind{4}{6}[\{3,5\}] \}. 
$$
Both $\mm_{1}$ and $\mm_{2}$ are $\{3,4\}$-homogeneous.  The toric fiber product
$\mm_{1} \times_{\{3,4\}} \mm_{2}$ is
$$
\{
\ind{1}{\{3,5,6\}}[\{2,4\}],  \ind{2}{\{4,5,6\}}[\{1,3\}], \ind{\{1,2,3\}}{5}[\{4,6\}],
$$
\[
  \ind{\{1,2,4\}}{6}[\{3,5\}] , \ind{\{1,2\}}{\{5,6\}}[\{3,4\}]  \}.
\]
\end{ex}


\subsection{Graphical conditional independence models}\label{sec:graphical-ci-models}
Our main motivation for toric fiber products together of CI-ideals comes from an
application to the global Markov condition in graphical models.  Let $G$ be a
simple undirected graph on the vertex set~$V$.

\begin{defn}
  The \emph{global Markov ideal} $J_{G}$ is the CI-ideal
  \[ 
  J_{G} = \sum_{\textrm{$C$ separates $A$ and $B$ in $G$}} J_{\ind{A}{B}[C]}.
  \]
\end{defn}

\begin{lemma}
  The global Markov ideal is a binomial ideal.
\end{lemma}

\begin{proof}
  If a statement is valid on $G$ but does not involve all vertices, then it is the
  consequence of a valid statement that does use all vertices.  Indeed, if $v\in
  V\setminus (A\cup B\cup C)$, then $v$ cannot be connected to both $A$ and $B$ as
  then $C$ would not separate. It is thus connected to at most one of them,
  say~$A$. In this case $\ind{A\cup\set{v}}{B}[C]$ is a valid statement
  for~$G$. Now use the decomposition property, also valid for CI-ideals,
  $\ind{A\cup\set{i}}{B}[C] \Rightarrow \ind{A}{B}[C]$ to get the result.
\end{proof}

Assume that we can decompose the vertex set of $G$ as $V = V_{1} \cup V_{2}$, such
that the induced subgraph on $S \defas V_{1}\cap V_{2}$ is complete, and any path
from $V_{1}$ to $V_{2}$ passes~$S$. In this case $S$ is a \emph{separator}. Since
a global Markov ideal is binomial it is $S$-homogeneous, and the same holds for
the CI-ideals $J_{G_{1}}$ and $J_{G_{2}}$, arising from the induced subgraphs on
$V_{1}$ and~$V_{2}$. 

\begin{thm} \label{thm:graphical-tfp} Let $G$ be a graph with vertex set $V =
  V_{1} \cup V_{2}$ and let $S = V_{1} \cap V_{2}$ be a separator.  Let $G_{1}$ and
  $G_{2}$ be the induced subgraphs of $G$ on vertex sets $V_{1}$ and~$V_{2}$.
  Then $J_{G}$ is the toric fiber product
  \[ J_{G} = J_{G_{1}} \tfp J_{G_{2}} \] where $\cala$ is the matrix associated to
  the $S$-grading.
\end{thm}

\begin{proof}
  It is easy to check that all CI-statements defining $J_{G_{1}} \tfp J_{G_{2}}$
  by Theorem~\ref{sec:citfp-thm} are valid on~$G$ and thus $J_{G} \supseteq
  J_{G_{1}} \tfp J_{G_{2}}$.  For the other containment let $\ind{A}{B}[C]$ be a
  an independence statement implied by the global Markov condition on~$G$ such
  that $A\cup B\cup C = V$.  If $A \subseteq V_{1}$, then $\ind{A}{B}[C]$ is
  implied by $\mathcal{L}_{1}$ since
  \[
  \ind{A } {B \setminus( V_{2} \setminus S) }[ C \setminus (V_{2} \setminus S) ]
  \]
  is a global Markov statement on~$G_{1}$.  After potentially replacing it by its
  symmetric version and lifting we find
  \[
  \ind{A} {B \cup( V_{2} \setminus S)}[ C \setminus (V_{2} \setminus S) ] \in
  \mathcal{L}_{1}.
  \]
  By the weak union property in Proposition~\ref{prop:CI-inference},
  $J_{\ind{A}{B}[C]}$ is contained in $J_{\mathcal{L}_{1}}$.  Note that if $B
  \setminus( V_{2} \setminus S) = \emptyset$, then the resulting CI-statement is
  implied by $\mathcal{S}$.  If $B \subset V_{1}$, $A \subset V_{2}$, or $B\subset
  V_{2}$ then $\ind{A}{B}[C]$ is similarly implied by $\mathcal{L}_{1}$ or
  $\mathcal{L}_{2}$.

  It remains to consider the case that both $A$ and $B$ have non-trivial
  intersection with both $V_{1}\setminus S$ and $V_{2} \setminus S$.  Since the
  subgraph induced on $S$ is complete, we can assume that $A \cap S = \emptyset$.
  Let $A_{i} = A \cap V_{i}$, $i=1,2$ and a binomial associated to $\ind{A}{B}[C]$
  has the form
\begin{equation*}
f = 
p_{i_{A_{1}}i_{A_{2}}i_{B} i_{C} } p_{j_{A_{1}}j_{A_{2}}j_{B} i_{C} } -
p_{i_{A_{1}}i_{A_{2}}j_{B} i_{C} } p_{j_{A_{1}}j_{A_{2}}i_{B} i_{C} }.
\end{equation*}
The independence statements
\begin{equation*}
  \ind{A_{1}}{B \cup A_{2}}[C] \mbox{ and  } \ind{A_{2}}{B\cup A_{1}}[C]
\end{equation*}
are both valid in $G$, since any path from $A_{1}$ to $A_{2}$ must traverse~$S$,
and all such paths are blocked either before they get to $B \cap S$, or at $C \cap
S$.  By the argument in the first paragraph of the proof, the first statement
belongs to $J_{\mathcal{L}_{1}}$ and the second statement belongs to
$J_{\mathcal{L}_{2}}$.  Together they imply $f \in J_{\mathcal{L}_{1}} +
J_{\mathcal{L}_{2}}$ since
\begin{eqnarray*}
f & = &  (p_{i_{A_{1}}i_{A_{2}}i_{B} i_{C} } p_{j_{A_{1}}j_{A_{2}}j_{B} i_{C} } -
p_{i_{A_{1}}j_{A_{2}}j_{B} i_{C} } p_{j_{A_{1}}i_{A_{2}}i_{B} i_{C} } )  \\
 &  & + (p_{i_{A_{1}}j_{A_{2}}j_{B} i_{C} } p_{j_{A_{1}}i_{A_{2}}i_{B} i_{C} } -
p_{i_{A_{1}}i_{A_{2}}j_{B} i_{C} } p_{j_{A_{1}}j_{A_{2}}i_{B} i_{C} }).
\end{eqnarray*}
Thus, all binomials from CI-statements implied by $G$ belong to $J_{G_{1}} \tfp
J_{G_{2}}$.
\end{proof}

As an immediate corollary we get the following known
result~\cite{Dobra,TakkenThesis}.
\begin{cor}
  \label{cor:chordalCIisPrime}
  The global Markov ideal of a chordal graph is prime.
\end{cor}
\begin{proof}
  A chordal graph decomposes as a product of its maximal cliques.  Inductively applying Theorem~\ref{thm:graphical-tfp} and the fact that the toric fiber product of geometrically prime ideals is geometrically prime, gives the result.
\end{proof}

The following corollary was one of our initial motivations for this section and
Theorem~\ref{thm:primary-decomp}.
 
\begin{cor}[Primary decompositions of graphical CI-ideals]
  Let $G$ be a graph with vertex set $V = V_{1} \cup V_{2}$ with $V_{1} \cap
  V_{2}$ a separator in~$G$.  Let $G_{1}$ and $G_{2}$ be the induced subgraphs on
  $V_{1}$ and $V_{2}$ respectively.  A primary decomposition of $J_{G}$ can be
  obtained from toric fiber products of the primary components of $J_{G_{1}}$
  and~$J_{G_{2}}$.
\end{cor}

As the primary decompositions of the CI-ideals $J_{G}$ are unknown for most
graphs, we do not know in which situations we can guarantee that the toric fiber
products of irredundant primary decompositions of CI-ideals yield an irredundant
primary decomposition.  In concrete situations
Corollary~\ref{cor:pers-prim-decomp-1} and Lemma~\ref{lem:two-gradings} can be
used.  For instance the primary decomposition of the chain of squares in
Example~\ref{ex:3squares} is irredundant.  Explicit computation shows that none of
the eight monomial minimal primes contains all monomials of a given multidegree,
and the same holds, of course, for the toric ideal. By
Corollary~\ref{cor:pers-prim-decomp-1} the toric fiber products of the prime
components yield an irredundant prime decomposition of the ideal of two squares
glued along an edge.  When gluing the next square the grading is different, but
Lemma~\ref{lem:two-gradings} guarantees that the hypothesis of
Corollary~\ref{cor:pers-prim-decomp-1} is still fulfilled.  Unfortunately this
argument cannot be applied to all conditional independence models, as the
following example demonstrates.

\begin{ex}
  Consider the binary graphical conditional independence model of the complete
  bipartite graph $K_{3,2}$, labeled such that $\{1,2,3\}$, and $\{4,5\}$ are
  independent sets.  The CI-ideal $J_{K_{3,2}}$ is radical as a computation with
  {\tt Binomials} shows~\cite{BinomialsM2}.  Consider the edge~$1$--$4$.  Its
  induced grading takes values in $\nn^{4}$.  The homogeneous elements
  \begin{gather*}
    p_{11111}p_{11212}p_{21122}p_{21221}-p_{11112}p_{11211}p_{21121}p_{21222}, \\
    p_{11121}p_{11222}p_{21112}p_{21211}-p_{11122}p_{11221}p_{21111}p_{21212}
  \end{gather*}
  witness minimal primes $P_{1}, P_{2}$ with the property that $(P_{1})_{\bfa} =
  \kk[p]_{\bfa}$ for all $\bfa \in \nn\{e_{12},e_{21}\}$ while $(P_{2})_{\bfa} =
  \kk[p]_{\bfa}$ for all $\bfa \in \nn\{e_{11},e_{22}\}$.  The prime decomposition
  of the toric fiber product, given by the toric fiber products of the minimal
  primes of two copies of $J_{K_{2,3}}$, has a component $P_{1}\times_{\cala}
  P_{2}$ which equals the maximal ideal of the fiber product's polynomial ring,
  and is thus redundant.
\end{ex}


\section*{Acknowledgements}
\label{sec:acknowledgement}

Alexander Engstr\"om gratefully acknowledges support from the Miller Institute for
Basic Research in Science at UC Berkeley.  Thomas Kahle was supported by an EPDI
Fellowship.  Seth Sullivant was partially supported by the David and Lucille
Packard Foundation and the US National Science Foundation (DMS 0954865).

The authors are happy to thank the Mittag--Leffler institute for hosting them for
the final part of this project, during the program on ``Algebraic Geometry with a
View towards Applications''.  Johannes Rauh made valuable comments on an earlier
version of the manuscript.



\begin{thebibliography}{99}

\bibitem{4ti2} 
{\tt 4ti2} Team. {\tt 4ti2} -- A software package for algebraic, geometric and combinatorial problems on linear spaces.  Available at {\tt www.4ti2.de}.

\bibitem{Aoki2003} 
Aoki, Satoshi; Takemura, Akimichi.
Minimal basis for a connected Markov chain over $3\times 3\times K$ contingency tables with fixed two-dimensional marginals.
\emph{Aust. N. Z. J. Stat.} {\bf 45} (2003), no.~2, 229--249. 

\bibitem{Buczynska2010} 
Buczy\'nska, Weronika.  
\emph{J. Algebraic Combin.} {\bf 35} (2012), no.~3, 421--460. 
 
\bibitem{Chen2006}
Chen, Yuguo; Dinwoodie, Ian H.; Sullivant, Seth.
Sequential importance sampling for multiway tables.
\emph{Ann. Statist.} {\bf 34} (2006), no.~1, 523--545. 
 
\bibitem{Develin2003}
Develin, Mike; Sullivant, Seth.
Markov bases of binary graph models.
\emph{Ann. Comb.} {\bf 7} (2003), no.~4, 441--466. 

\bibitem{Diaconis1998}  Diaconis, Persi; Sturmfels, Bernd.  
Algebraic algorithms for sampling from conditional distributions. 
 \emph{Ann. Statist.}  {\bf 26} (1998),  no.~1, 363--397.

\bibitem{Diestel2006}
Diestel, Reinhard.
\emph{Graph theory.} 
Third edition. Graduate Texts in Mathematics, 173. Springer-Verlag, Berlin, 2005. 411 pp.

\bibitem{Dobra} 
Dobra, Adrian.
Markov bases for decomposable graphical models.
\emph{Bernoulli} {\bf 9} (2003), no.~6, 1093--1108. 
  
\bibitem{Dobra2004} 
Dobra, Adrian; Sullivant, Seth.
A divide-and-conquer algorithm for generating Markov bases of multi-way tables.
\emph{Comput. Statist.} {\bf 19} (2004), no.~3, 347--366. 

\bibitem{Drton2009}
Drton, Mathias; Sturmfels, Bernd; Sullivant, Seth.
\emph{Lectures on algebraic statistics.} 
Oberwolfach Seminars, 39. \emph{Birkh\"auser Verlag, Basel,} 2009. viii+171 pp.

\bibitem{Engstrom2008}  
Engstr\"om, Alexander.
Cut ideals of $K_4$-minor free graphs are generated by quadrics. 
\emph{Michigan Math. J.} {\bf 60} (2011), no 3. 

\bibitem{Engstrom2010} 
Engstr\"om, Alexander; Nor\'en, Patrik.
Ideals of graph homomorphisms. \emph{Ann. Comb.} {\bf 17} (2013), no. 1, 71--103.

\bibitem{Garcia2005}
Garcia, Luis David; Stillman, Michael; Sturmfels, Bernd.
Algebraic geometry of Bayesian networks. 
\emph{J. Symbolic Comput.} {\bf 39} (2005), no.~3-4, 331--355. 

\bibitem{Geiger1993}
Geiger, Dan; Pearl, Judea.
Logical and Algorithmic Properties of Conditional
Independence and Graphical Models.
 \emph{Ann. Statist.}  {\bf 21} (1993),  no.~4, 2001--2021.

\bibitem{Hillar2009}
Hillar, Christopher J.; Sullivant, Seth.
Finite Gr\"obner bases in infinite dimensional polynomial rings and applications. 
\emph{Adv. Math.} {\bf 229} (2012), no.~1, 1--25.

\bibitem{Hosten2002}
Hosten, Serkan; Sullivant, Seth.
Gr\"obner bases and polyhedral geometry of reducible and cyclic models.
\emph{J. Combin. Theory Ser. A} {\bf 100} (2002), no.~2, 277--301. 

\bibitem{Hosten2007}
Hosten, Serkan; Sullivant, Seth.
A finiteness theorem for Markov bases of hierarchical models.
\emph{J. Combin. Theory Ser. A} {\bf 114} (2007), no.~2, 311--321. 

\bibitem{BinomialsM2}
Kahle, Thomas. Decompositions of binomial ideals, \emph{J. Software
  for Algebraic Geometry} \textbf{4} (2012), 1--5.

\bibitem{Kahle2008} 
Kahle, Thomas; Rauh, Johannes.  
Markov Bases Database. {\tt http://markov-bases.de/}.

\bibitem{KRS12} 
Kahle, Thomas; Rauh, Johannes; Sullivant, Seth.
Positive Margins and Primary decomposition.
\emph{J. Commutative Algebra}, to appear. {\tt arxiv:1201.2591}.

\bibitem{Kral2008} 
Kr\'al, Daniel; Norine, Serguei; Pangr\'ac, Ond\v{r}ej.
Markov bases of binary graph models of $K_4$-minor free graphs. 
\emph{J. Combin. Theory Ser. A} {\bf 117} (2010), no.~6, 759--765. 

\bibitem{Lauritzen1996} 
Lauritzen, Steffen L.
\emph{Graphical models.} 
Oxford Statistical Science Series, 17. Oxford Science Publications. \emph{The Clarendon Press, Oxford University Press, New York,} 1996. x+298 pp.

\bibitem{Manon2009} 
Manon, Christopher A.
The algebra of conformal blocks. 
(2009), {\tt arxiv:0910.0577}.

\bibitem{Michalek2010}
Micha\l ek, Mateusz.
Geometry of phylogenetic group-based models.
\emph{J. Algebra} {\bf 339} (2011), 339--356. 

\bibitem{Mumford1994} 
Mumford, David; Fogarty, John; Kirwan, Frances.
\emph{Geometric invariant theory.} 
Third edition. Ergebnisse der Mathematik und ihrer Grenzgebiete (2), 34. \emph{Springer-Verlag, Berlin,} 1994. xiv+292 pp.

\bibitem{Noren2012}
Nor\'en, Patrik. 
The three-state toric homogeneous Markov chain model has Markov degree two. (2012), {\tt arxiv:1207.0077}.

\bibitem{Ohsugi2010}
Ohsugi, Hidefumi.
Normality of cut polytopes of graphs in a minor closed property.
\emph{Discrete Math.} {\bf 310} (2010), no.~6-7, 1160--1166. 

\bibitem{Petrovic2009} 
Petrovi\'c, Sonja; Stokes, Eric.
Betti numbers of Stanley-Reisner rings determine hierarchical Markov degrees.
\emph{J. Algebraic Combin.}, to appear. {\tt arxiv:0910.1610.}

\bibitem{Robertson2004}
Robertson, Neil; Seymour, Paul D.
Graph minors. XX. Wagner's conjecture. 
\emph{J. Combin. Theory Ser. B} {\bf 92} (2004), no.~2, 325--357. 

\bibitem{Simis2000} 
Simis, Aron; Ulrich, Bernd.
On the ideal of an embedded join. 
\emph{J. Algebra} {\bf 226} (2000), no.~1, 1--14. 

\bibitem{Sturmfels2008}
Sturmfels, Bernd; Sullivant, Seth.
Toric geometry of cuts and splits. 
\emph{Michigan Math. J.} {\bf 57} (2008), 689--709. 

\bibitem{Sturmfels2011}
Sturmfels, Bernd; Welker, Volkmar.
Commutative algebra of statistical ranking. 
\emph{J. Algebra}, {\bf 361} (2012), 264--286.

\bibitem{Sullivant2010} 
Sullivant, Seth.
Normal binary graph models.
\emph{Ann. Inst. Statist. Math.} {\bf 64} (2010), no.4, 717--726. 

\bibitem{Sullivant2007}
Sullivant, Seth.
Toric fiber products. 
\emph{J. Algebra} {\bf 316} (2007), no.~2, 560--577. 

\bibitem{SwaHu}
Swanson, Irena; Huneke, Craig.
\emph{Integral Closure of Ideals, Rings, and Modules.}
LMS Lecture note series \emph{Cambridge University Press, Cambridge} 2006.

\bibitem{Takemura2008} 
Takemura, Akimichi; Thomas, Patrick; Yoshida, Ruriko. 
Holes in semigroups and their applications to the two-way common diagonal effect model. 
In: \emph{Proceedings of the 2008 International Conference on Information Theory
  and Statistical Learning}, ITSL 2008, CSREA Press, 2008, 67--72.

\bibitem{TakkenThesis} 
Takken, Asya.
\emph{Monte Carlo goodness-of-fit tests for discrete data.} 
Ph.D. dissertation, Dept. Statistics, Stanford Univ, 2000.

\bibitem{Tousi2003}
Tousi, Masoud; Yassemi, Siamak.
Tensor products of some special rings.
\emph{J. Algebra} {\bf 268} (2003), no.~2, 672--676.

\end{thebibliography}
\end{document}